\documentclass[table]{amsart} 
\usepackage{amsmath}
\usepackage[margin=1.0in]{geometry}
\usepackage{amscd,amsmath,amsxtra,amsthm,amssymb,stmaryrd,xr,mathrsfs,mathtools,enumerate,commath, comment, enumitem, graphicx}
\usepackage{amsfonts, amssymb, amsthm, amsmath, braket, hyperref, mathtools, comment, mathrsfs, enumitem, cleveref}
\usepackage{stmaryrd}
\usepackage{amsfonts}
\usepackage{orcidlink}
\usepackage{multirow}
\usepackage{xcolor}
\usepackage{commath}
\usepackage{float}
\usepackage{comment}
\usepackage{tikz-cd}
\usepackage{tkz-graph}
\usepackage{colonequals}
\usepackage{hyperref}
\usepackage{graphicx}
\usepackage{bigstrut}
\usepackage{tabularx}
\usepackage{longtable, multirow, array}

\usepackage{longtable} 
\usepackage{pdflscape} 
\usepackage{booktabs}
\usepackage{hyperref}
\definecolor{vegasgold}{rgb}{0.77, 0.7, 0.35}
\definecolor{darkgoldenrod}{rgb}{0.72, 0.53, 0.04}
\definecolor{gold(metallic)}{rgb}{0.83, 0.69, 0.22}
\hypersetup{
 colorlinks=true,
 linkcolor=darkgoldenrod,
 filecolor=brown,      
 urlcolor=gold(metallic),
 citecolor=darkgoldenrod,
 }

\usepackage[all,cmtip]{xy}
\DeclareFontFamily{U}{wncy}{}
\DeclareFontShape{U}{wncy}{m}{n}{<->wncyr10}{}
\DeclareSymbolFont{mcy}{U}{wncy}{m}{n}
\DeclareMathSymbol{\Sh}{\mathord}{mcy}{"58}
\usepackage[T2A,T1]{fontenc}
\usepackage[OT2,T1]{fontenc}

\usepackage{tikz}
\usetikzlibrary{shapes.geometric}
\usetikzlibrary{decorations.markings}
\tikzset{every loop/.style={min distance=10mm,looseness=10}}
\tikzstyle{vertex}=[auto=left,circle,minimum size=1pt,inner sep=0pt]

\newtheorem{theorem}{Theorem}[section]
\newtheorem{lemma}[theorem]{Lemma}

\newtheorem*{theorem*}{Theorem}
\newtheorem*{ass*}{Assumption}
\newtheorem{definition}[theorem]{Definition}
\newtheorem{corollary}[theorem]{Corollary}
\newtheorem{remark}[theorem]{Remark}

\newtheorem{proposition}[theorem]{Proposition}

\newcommand{\cF}{\mathcal{F}}

\newcommand{\cH}{\mathcal{H}}

\newcommand{\Z}{\mathbb{Z}}

\newcommand{\Q}{\mathbb{Q}}

\newcommand{\cC}{\mathcal{C}}
\newcommand{\cO}{\mathcal{O}}
\newcommand{\cS}{\mathcal{S}}

\newcommand{\op}[1]{\operatorname{#1}}

\newcommand\mtx[4] { \left( {\begin{array}{cc}
 #1 & #2 \\
 #3 & #4 \\
 \end{array} } \right)}
\newcommand\dmtx[3] { \left( {\begin{array}{ccc}
 #1 &  & \\
 & #2 & \\
 &  & #3 \\
 \end{array} } \right)}

\numberwithin{equation}{section}

\begin{document}

\title[Shapes of pure quartic number fields]{On the distribution of shapes of pure quartic number fields}

\author[S.~Das]{Sudipa Das\, \orcidlink{0009-0001-9133-019X}}
\address[Das]{Department of Mathematics, Harish-Chandra Research Institute, A CI of Homi Bhabha National Institute, Chhatnag Rd, Jhusi, Prayagraj, Uttar Pradesh 211019, India}
\email{sudipadas@hri.res.in}

\author[S.~Kala]{Sushant Kala\, \orcidlink{0009-0004-7689-0131}}
\address[Kala]{Chennai Mathematical Institute, H1, SIPCOT IT Park, Kelambakkam, Siruseri, Tamil Nadu 603103, India}
\email{sushantkala@cmi.ac.in}

\author[A.~Mukhopadhyay]{Arunabha Mukhopadhyay\, \orcidlink{0009-0002-2256-6366}}
\address[Mukhopadhyay]{Department of Mathematics, Institute of Mathematical Sciences (HBNI), CIT Campus, IV Cross
Road, Chennai, India-600113}
\email{arunabham@imsc.res.in}

\author[A.~Ray]{Anwesh Ray\, \orcidlink{0000-0001-6946-1559}}
\address[Ray]{Chennai Mathematical Institute, H1, SIPCOT IT Park, Kelambakkam, Siruseri, Tamil Nadu 603103, India}
\email{anwesh@cmi.ac.in}

\keywords{shapes of number fields, equidistribution, arithmetic statistics}
\subjclass[2020]{11R16, 11R45 (primary), 11E12, 11P21 (secondary)}

\maketitle

\begin{abstract}
The \emph{shape} of a number field is a subtle arithmetic invariant arising from the geometry of numbers. It is defined as the equivalence class of the lattice of integers with respect to linear operations that are composites of rotations, reflections, and positive scalar dilations. For a number field of degree $n$, the shape is a point in the \emph{space of shapes} $\mathcal{S}_{n-1}$, which is the double quotient $\op{GL}_{n-1}(\mathbb{Z}) \backslash \op{GL}_{n-1}(\mathbb{R}) / \op{GO}_{n-1}(\mathbb{R})$. In this paper, we investigate the distribution of shapes in the family of \emph{pure quartic fields} $K_m = \mathbb{Q}(\sqrt[4]{m})$. We prove that the shape of $K_m$ lies on one of $5$ explicitly described torus orbits in $\mathcal{S}_3$, determined by the sign and residue class of $m \bmod 32$. It is shown that the shape on a given torus orbit is completely determined by two parameters, one of which varies continuously, while the other takes values in a discrete set. As a result, the distribution of shapes in this family is governed by a product of a continuous and a discrete measure. Our results shed new light on a question posed by Manjul Bhargava and Piper H concerning the distribution of shapes in families of non-generic number fields of fixed degree. Notably, the limiting distribution in our case does \emph{not} arise as the restriction of the natural measure on $\mathcal{S}_3$ induced by Haar measure on $\op{GL}_3(\mathbb{R})$.
\end{abstract}

\section{Introduction}

\subsection{Motivation and historical context}
 The shape of a number field is a subtle arithmetic invariant which arises from the geometry of numbers. Given an inner product space $V\simeq \mathbb{R}^r$, the \emph{shape} of a full rank lattice $\Lambda\subset V$ is an invariant which is unchanged by rotation, reflection and positive scalar dilations. Formally, the space of all such shapes is given by the double coset space
$$
\mathcal{S}_r := \operatorname{GL}_r(\mathbb{Z}) \backslash \operatorname{GL}_r(\mathbb{R}) / \operatorname{GO}_r(\mathbb{R}),
$$
\noindent where $\operatorname{GO}_r(\mathbb{R}) \subset \operatorname{GL}_r(\mathbb{R})$ denotes the group of matrices that preserve the standard inner product up to a positive scalar multiple.
Let $K$ be a number field of degree $n$ and ring of integers $\cO_K$. One can embed $K$ into its Minkowski space 
 \[j: K\rightarrow K\otimes_{\Q} \mathbb{R}\simeq \mathbb{R}^n\] and thus associate a lattice $j(\cO_K)\subset \mathbb{R}^n$ to $K$. Here, $K\otimes_{\Q} \mathbb{R}$ is endowed with a natural inner product which arises from the trace form on $K$. The \emph{shape} $K$ is defined to be the shape of orthogonal projection of $j(\cO_K)$ onto the hyperplane perpendicular to the vector $j(1)$. 
\par All quadratic number fields have the same shape since all rank $1$ lattices are dilations of a fixed lattice. For cubic number fields $K$, the associated rank $2$ lattice is represented by a point in the complex upper half plane $\mathbb{H}$. The shape of $K$ is represented by a point in the fundamental domain \[\cF_2=\{x+iy\in \mathbb{H}\mid x\in [0, 1/2), \quad x^2+y^2\geq 1\},\] which comes equipped with the hyperbolic measure, which is induced from the Haar measure on $\op{GL}_2(\mathbb{R})$. 
D.~Terr first introduced the study of the shape of a number field in his PhD thesis \cite{Terr97}, written under the supervision of H.~Lenstra. He proved that the shapes of cubic fields are equidistributed in the space of shapes of rank-$2$ lattices. Furthermore, he showed that the shape associated to any Galois cubic field is always hexagonal. M.~Bhargava and A.~Shnidman \cite{BS14} classified the shapes of cubic orders whose automorphism group is isomorphic to $\mathbb{Z}/3\mathbb{Z}$, and also established asymptotics for the number of such number fields. In a related development, G.~Mantilla-Soler and M.~Monsurró \cite{MSM16} determined the shapes of cyclic number fields of prime degree $\ell$. 

\par M.~Bhargava and Piper~H. extended Terr's result to $S_n$ number fields of degree $n=4,5$ and conjectured that such equidistribution results hold for all $n\geq 6$ as well. In greater detail, let $\mu$ denote the measure on $\mathcal{S}_{n-1}$ induced from the Haar measure on $\op{GL}_{n-1}(\mathbb{R})$. An $n$-ic field is a number field $K/\Q$ of degree $n$. Given an $n$-ic field $K$,  denote by $\widetilde{K}$ its Galois closure over $\Q$ and set 
\[\op{Gal}(K/\Q):=\op{Gal}(\widetilde{K}/\Q).\] We note that $\op{Gal}(K/\Q)$ acts faithfully on the $n$-embeddings $K\hookrightarrow \mathbb{C}$ and may thus be viewed as a subgroup of $S_n$. For $n=3,4,5$ and a positive real number $X$, let $N_n^{(i)}(X)$ be the number of $n$-ic extensions of $\Q$ with $i$ pairs of complex embeddings such that $\op{Gal}(K/\Q)\simeq S_n$ and the absolute discriminant of $K$ is less than $X$. Given a measurable set $Y$ in $\mathcal{S}_{n-1}$ whose boundary has measure zero, let $N_n^{(i)}(X, Y)$ be the number of $n$-ic extensions of $\Q$ with $i$ pairs of complex embeddings such that $\op{Gal}(K/\Q)\simeq S_n$ and the absolute discriminant of $K$ is less than $X$, and whose ring of integers has shape contained in $Y$. Then the main result in \emph{loc. cit.} establishes that 
\[\lim_{X\rightarrow \infty} \frac{N_n^{(i)}(X, Y)}{N_n^{(i)}(X)}=\frac{\mu(Y)}{\mu(\mathcal{S}_{n-1})}. \] The proof of this result relies on the existence of parameterizations of cubic, quartic and quintic orders by Delone, Faddeev \cite{DeloneFaddeev} and Bhargava \cite{HighercompositionIII, HighercompositionIV}. Such parameterizations are not known in cases when $n\geq 6$.
\par In \cite{BH16}, the authors do mention that: 
"\emph{It is an interesting problem to determine the distribution of lattice shapes for $n$-ic number fields having a given non-generic (i.e., non-$S_n$) associated Galois group, even heuristically.}" It is a natural question to study the distribution of shapes of families of non-$S_n$ number fields $K/\Q$ of degree $n$. One natural family to consider is the family of \emph{pure number fields} of degree $n$, namely, number fields of the form $K=\Q(m^{1/n})$, where $m$ is an $n$-power free non-zero integer. Alternatively, this is the family of number fields $K/\Q$ of degree $n$, with Galois group $\op{Gal}(\widetilde{K}/\Q)\simeq (\Z/n\Z)^\times \ltimes (\Z/n\Z)$ and having resolvent field $\Q(\mu_n)$. This question has been studied for pure cubic fields by R.~Harron \cite{Har17}, and for pure prime degree number fields by E.~Holmes \cite{Hol22}. \par Let $m$ be a positive cubefree integer and $\sqrt[3]{m}$ be the real third root of $m$ and set $K_m:= \Q(\sqrt[3]{m})$. These fields naturally fall into two families, distinguished by whether the prime $(3)$ is tamely or wildly ramified in $K_m$. All other primes are either unramified or tamely ramified. Write $m = ab^2$, where $a$ and $b$ are coprime, squarefree, positive integers. Note that $K_m = \Q(\sqrt[3]{m}) = \Q(\sqrt[3]{m'})$, where $m' := a^2b$, since $mm' = (ab)^3$ is a cube. Assume without loss of generality, that $a > b$. Define the \emph{ratio} of $K_m$ to be $r_m := a/b$. Say that $K_m$ is of \emph{Type \rm{I}} if $m \not\equiv \pm1 \pmod{9}$, and of \emph{Type \rm{II}} otherwise. Note that $K_m$ is tamely ramified (or unramified) at $(3)$ if and only if it is of Type \rm{I}. Let $\alpha$ and $\beta$ be the roots in $K$ of $x^{3}-m$ and $x^{3}-m^{\prime}$, respectively, so that $\beta=\alpha^{2} / b, \alpha=\beta^{2} / a$, and $K=\Q(\alpha)=\Q(\beta)$. The discriminant of $K$ is $-3^{3} a^{2} b^{2}$ (resp. $-3 a^{2} b^{2}$) for Type \rm{I} (resp. Type \rm{II}). For $m \equiv \pm 1(\bmod 9)$, the element $\nu=\left(1 \pm \alpha+\alpha^{2}\right) / 3$ is in $\cO_{K_m}$. An integral basis of $K_m$ is given by:
\[\begin{split}
    \cO_{K_m}=\begin{cases}
        & \langle 1, \alpha, \beta\rangle \quad \text{ if }K_m\text{ is of Type \rm{I}},\\
       & \langle 1, \nu, \beta\rangle \quad \text{ if }K_m\text{ is of Type \rm{II}},\\
    \end{cases}
\end{split}\]
see \cite[Lemma 2.1]{Har17}.

\begin{theorem}[Theorem A \cite{Har17}]
The shapes of pure cubic fields fall into two distinct loci:
\begin{itemize}
\item If $K$ is of Type \rm{I}, its shape lies on the imaginary axis within $\mathcal{F}$, and is given explicitly by $i r_K^{1/3} \in \mathcal{F}$.
\item If $K$ is of Type \rm{II}, its shape lies on the vertical line $\operatorname{Re}(z) = 1/3$, above height $1/3$; more precisely, the shape is $\frac{1 + i r_K^{1/3}}{3} \in \mathbb{H}$.
\end{itemize}
\end{theorem}

To study the distribution of these shapes, one introduces measures on the sets $\mathscr{S}_{\mathrm{I}} := \{i y : y \geq 1\}$ and $\mathscr{S}_{\mathrm{II}} := \left\{ \frac{1 + i y}{3} : y \geq 1 \right\}$. A natural choice is the measure induced by the hyperbolic metric on $\mathbb{H}$, or equivalently, the unique (up to scaling) measure invariant under the action of a subgroup of $\mathrm{SL}_2(\mathbb{R})$. Since the hyperbolic line element on $\mathbb{H}$ is

$$
ds = \frac{\sqrt{dx^2 + dy^2}}{y},
$$
\noindent this induces the measure $\frac{dy}{y}$ on both $\mathscr{S}_{\mathrm{I}}$ and $\mathscr{S}_{\mathrm{II}}$. \par Alternatively, observe that the diagonal torus \[\left\{ \mtx{y}{0}{0}{y^{-1}}\right\} \subseteq \mathrm{SL}_2(\mathbb{R}),\]acting on the base point $i \in \mathbb{H}$, defines a homeomorphism onto the positive imaginary axis, sending the subset with $y > 1$ onto $\mathscr{S}_{\mathrm{I}}$. This identifies $\mathscr{S}_{\mathrm{I}}$ with $(1, \infty)$ via $y \mapsto i y$, and thus endows it with the measure $\frac{dy}{y}$ induced from Haar measure on the torus. Conjugating the torus by an element of $\mathrm{SL}_2(\mathbb{R})$ yields the same construction for $\mathscr{S}_{\mathrm{II}}$, giving rise to an invariant measure there as well. Denote the resulting measures on $\mathscr{S}_{\mathrm{I}}$ and $\mathscr{S}_{\mathrm{II}}$ by $\mu_{\mathrm{I}}$ and $\mu_{\mathrm{II}}$, respectively. For $\ast\in \{\rm{I}, \rm{II}\}$, set 
\[[R_1, R_2)_\ast=\begin{cases}
    i[R_1, R_2)&\text{ if }\ast=\rm{I},\\
    \left(\frac{1+i[R_1, R_2)}{3}\right)&\text{ if }\ast=\rm{II}.\\
\end{cases}\]Define
$$
c_{\mathrm{I}}=\frac{2 c \sqrt{3}}{15} \quad \text { and } \quad c_{\mathrm{II}}=\frac{c \sqrt{3}}{10}
$$
where
$$
c:=\prod_{p}\left(1-\frac{3}{p^{2}}+\frac{2}{p^{3}}\right)
$$
the product being over all primes $p$.

\begin{theorem}[Theorem C \cite{Har17}] Then, for all $R_{1}, R_{2}$,
$$
\lim _{X \rightarrow \infty} \frac{\#\left\{K \text { of type }\ast\mid \,|\Delta_K| \leq X,\,\operatorname{sh}_K \in\left[R_{1}, R_{2}\right)_\ast\right\}}{c_{\ast} \sqrt{X}}=\int_{\left[R_{1}, R_{2}\right)_\ast} d \mu_{\ast}
$$
where $\Delta(K)$ is the discriminant of $K$ and $\operatorname{sh}(K)$ is the shape of $K$ (taken in $\mathscr{S}_{\ast}$).
\end{theorem}
\par More recently, Harron \cite{HarronANT} has proven distribution results for complex quadratic cubic fields with fixed quadratic resolvent. In this case, the shape lies on a geodesic of the modular surface defined by the field's trace zero norm. Moreover, the shapes are equi-distributed with respect to the hyperbolic measure.

\subsection{Main results} In this article, we study the distribution of shapes of the rings of integers of pure quartic number fields $K_m:=\Q(\sqrt[4]{m})$, where $m$ is a non-zero fourth power free and $m\neq -4$. For this family of number fields, the integral basis and discriminant were computed by Funakura \cite{Funakura}, and the classification gives five families, see Theorem \ref{IB funakura}. Suppose that $m$ has sign $\pm$, then we say that $m$ is of:
\begin{enumerate}
    \item Type $\rm{I}\pm$: if $m\equiv 1\pmod{8}$, 
    \item Type $\rm{II}\pm$: if $m\equiv 2,3\pmod{4}$,
    \item Type $\rm{III}\pm$: if $m\equiv 4\pmod{16}$ or $m\equiv 5\pmod{8}$,
    \item Type $\rm{IV}\pm$: if $m\equiv 12\pmod{32}$, 
    \item Type $\rm{V}\pm$: if $m\equiv 28\pmod{32}$.
\end{enumerate}
We write $m=ab^2c^3$ where $a,b,c$ are squarefree and coprime integers. Moreover, assume that $|a|\geq c$ and that $a<0$ if and only if $m$ is negative. For each type, we associate an explicit matrix:
\[
C_{\mathrm{I}} = \begin{pmatrix}
1 & 0 & 0 & 0 \\
\frac{1}{2} & \frac{1}{2} & 0 & 0 \\
0 & 0 & 1 & 0 \\
\frac{ab}{4} & \frac{1}{4} & \frac{b}{4} & \frac{1}{4}
\end{pmatrix},\quad C_{\mathrm{II}} = \begin{pmatrix}
1 & 0 & 0 & 0 \\
0& 1 & 0 & 0 \\
0 & 0 & 1 & 0 \\
0 & 0 & 0 & 1
\end{pmatrix}, \\ \quad
C_{\mathrm{III}} = \begin{pmatrix}
1 & 0 & 0 & 0 \\
\frac{1}{2} & \frac{1}{2} & 0 & 0 \\
0 & 0 & 1 & 0 \\
0 & 0 & \frac{1}{2} & \frac{1}{2}
\end{pmatrix}, \]\[
C_{\mathrm{IV}} = \begin{pmatrix}
1 & 0 & 0 & 0 \\
0 & 1 & 0 & 0 \\
\frac{1}{2} & \frac{1}{2} & \frac{1}{2} & 0 \\
0 & 0 & \frac{1}{2} & \frac{1}{2}
\end{pmatrix}, \quad
C_{\mathrm{V}} = \begin{pmatrix}
1 & 0 & 0 & 0 \\
0 & 1 & 0 & 0 \\
\frac{1}{2} & \frac{1}{2} & \frac{1}{2} & 0 \\
0 & \frac{1}{2} & \frac{b}{8} & \frac{1}{4}
\end{pmatrix}.
\]
\noindent The shape of $K_m$ is then given by the projection of the matrix
$$
C_\ast \begin{pmatrix}
1 & 0 & 0 & 0 \\
0 & \sqrt{c/|a|} & 0 & 0 \\
0 & 0 & 1/b & 0 \\
0 & 0 & 0 & \sqrt{|a|/c}
\end{pmatrix} C_\ast^T
$$
\noindent to the orthogonal complement of the vector $(1, 1, 1, 1)$, where $C_\ast$ is the matrix corresponding to the type of $K_m$. 
It thus turns out that the shape lies on one of $10$ torus orbits depending on the sign and Type of $m$. In all cases, the shape is determined by the \emph{shape parameters} 
\[\lambda(m)=(\lambda_1(m), \lambda_2(m)):=\left( \sqrt{\frac{c}{|a|}}, \frac{1}{b}\right),\] and thus determined by the pair $(\frac{|a|}{c}, b)\in [1, \infty)^2$. We note that unlike the shape parameters considered in \cite{Har17} and \cite{Hol22}, the parameter $\lambda_2(m)=\frac{1}{b}$ takes on a discrete set of values with an accumulation point at $0$. 

Harron shows that the shape is a complete invariant for the family of pure cubic fields (cf.~\cite[Theorem B]{Har17}). In particular, if $\Q(\!\sqrt[3]{m_1})$ and $\Q(\!\sqrt[3]{m_2})$ are pure cubic fields whose rings of integers have the same shape, then $\Q(\!\sqrt[3]{m_1})=\Q(\!\sqrt[3]{m_2})$. A similar statement holds for pure quartic fields, with a small caveat.

\begin{theorem}\label{shape complete invariant}
Let $m_1$ and $m_2$ be fourth-power--free integers such that $\operatorname{sh}(K_{m_1})=\operatorname{sh}(K_{m_2})$. Then exactly one of the following holds:
\begin{enumerate}
    \item $m_1=-m_2$ and $m_1\equiv m_2\equiv 2 \pmod{4}$,
    \item $m_1=m_2$.
\end{enumerate}
\end{theorem}

\par Our main result shows that there is a measure that gives the distribution of shapes, which is in fact the product of a discrete measure and a continuous measure. Define
$$
\psi(x) := \sum_{n = 1}^{\lfloor x \rfloor} \alpha(n),
\quad \text{where} \quad
\alpha(n) := \begin{cases}
n^{-2/3} \displaystyle\prod_{\ell \mid n} \left( \dfrac{\ell}{\ell + 2} \right) & \text{if } n \text{ is squarefree}, \\
0 & \text{otherwise},
\end{cases}
$$
\noindent and the product is over prime divisors $\ell$ of $n$. Let $\mu_\alpha(x) dx:=\sum_{n} \alpha(n)\delta_n(x)dx$ be the sum of Dirac measures supported only at squarefree positive integers. Let $\hat{\mu}$ denote the measure $\frac{dx_1}{x_1} \mu_\alpha(x_2)dx_2$ on $[1, \infty)^2$ and define $\psi_{\ast}(x):=\sum_{n=1}^{\lfloor x \rfloor} M_{\ast}(n) \alpha(n)$. Below is the main result of the article. \begin{theorem}\label{Main theorem}
Let $* \in \{\mathrm{I},\mathrm{II}, \mathrm{III}, \mathrm{IV}, \mathrm{V}\}$ and let $\mathcal{R}:=[R_1', R_1)\times [R_2', R_2)
\subset [1, \infty)^2$. Then, we have that
\[
\begin{split}& \lim_{X \to \infty}
\left(
\frac{
\#\left\{ K_m \mid \pm m>0,\, |\Delta_m| \leq X,\; K_m \text{ is of Type } \ast,\, (|a|/c,b) \in \mathcal{R} \right\}
}{
2^{r_*/3} X^{1/3}
}
\right)\\
= &
\frac{1}{2} \prod_\ell\left(1-\frac{3}{\ell^2}+\frac{2}{\ell^3}\right)(\log R_1 - \log R_1') \left( \psi_{*}(R_2) - \psi_{*}(R_2') \right) \\
= &
\frac{1}{2} \prod_\ell\left(1-\frac{3}{\ell^2}+\frac{2}{\ell^3}\right)\int_{\mathcal{R}} M_{*}(\lfloor x \rfloor) \, d\hat{\mu}(x),
\end{split}
\]
where
\[\begin{split}
& M_{\rm{I}}(b):=\frac{\# \{(\bar{a}, \bar{c})\in (\Z/8\Z)^2\mid \bar{a}\bar{b}^2\bar{c}^3=1\pmod{8}\}}{8^2}, \\
& M_{\rm{II}}(b):= \frac{\# \{(\bar{a}, \bar{c})\in (\Z/4\Z)^2\mid \bar{a}\bar{b}^2\bar{c}^3=2,3\pmod{4}\}}{16}, \\
& M_{\rm{III}}(b):=\frac{\# \{(\bar{a}, \bar{c})\in (\Z/16\Z)^2\mid \bar{a}\bar{b}^2\bar{c}^3=4, 5, 10\pmod{16}\}}{16^2}, \\
& M_{\rm{IV}}(b):=\frac{\# \{(\bar{a}, \bar{c})\in (\Z/32\Z)^2\mid \bar{a}\bar{b}^2\bar{c}^3=12\pmod{32}\}}{32^2}, \\
& M_{\rm{V}}(b):=\frac{\# \{(\bar{a}, \bar{c})\in (\Z/32\Z)^2\mid \bar{a}\bar{b}^2\bar{c}^3=28\pmod{32}\}}{32^2},
\end{split}\]
and
\[r_\ast=\begin{cases}
    2 & \text{ if }\ast = \mathrm{I}\text{ or }\mathrm{V},\\
    4& \text{ if }\ast = \mathrm{III}\text{ or }\mathrm{IV},\\
    8 & \text{ if }\ast=\mathrm{II}.
\end{cases}\]

\end{theorem}

\subsection{Organization}
Including the introduction, this article is divided into four sections. In \S\ref{sec1}, we review foundational concepts related to the shape of a number field. In \S\ref{sec3}, we compute the Gram matrices associated to pure quartic fields, using the explicit integral bases constructed in \cite{Funakura}. This leads to a complete classification of the shapes of pure quartic fields into ten distinct classes. We give a proof of Theorem \ref{shape complete invariant} at the end of this section. The main distribution results are established in \S\ref{sec4}. We begin by revisiting the notion of shape parameters and the Principle of Lipschitz. These, together with sieve methods, are then employed to estimate the number of integer lattice points in suitable regions. The section concludes with the proof of Theorem~\ref{Main theorem}.

\medskip
\subsection*{Acknowledgments} This project is the first of two on the theme of "shapes of number fields", both initiated during the \href{https://sites.google.com/view/imscnumbertheorygroupmeeting/home}{Number theory working group meeting} held at the Institute of Mathematical Sciences from 27th to 31st March, 2025. The fourth author gratefully acknowledges the opportunity to lead this project during the meeting and thanks the institute for providing a stimulating and collaborative research environment. The authors thank Sudipa Mondal for valuable discussions during the preparatory stages of this project. The authors also thank the anonymous referee for the helpful report.

\subsection*{Data Availability} There is no data associated to the results of this manuscript.

\section{Shapes of number fields}\label{sec1}
\par This section introduces notation and reviews background material on the shapes of number fields. Fix an integer $r \geq 2$, and let $M_r(\mathbb{R})$ denote the space of real $r \times r$ matrices. The group $\operatorname{GL}_r(\mathbb{R})$ consists of all invertible elements of $M_r(\mathbb{R})$, and denote by $\operatorname{I}_r$ the identity matrix. For any matrix $M \in M_r(\mathbb{R})$, write $M^T$ for its transpose. Let $V \simeq \mathbb{R}^r$ be an inner product space with standard orthonormal basis $\{e_1, \dots, e_r\}$, so that \[(e_i, e_j) = \delta_{i,j}=\begin{cases}
    1 & \text{ if }i=j,\\
    0 & \text{ if }i\neq j.
\end{cases}\] Let $\operatorname{GO}_r(\mathbb{R})$ be the subgroup of $\operatorname{GL}_r(\mathbb{R})$ consisting of matrices $M$ satisfying $M^T M = \lambda \operatorname{I}_r$ for some scalar $\lambda \in \mathbb{R}_{>0}$; these are precisely the invertible linear transformations preserving the inner product up to a positive scaling.

\begin{definition}
Let $\Lambda \subset V$ be a lattice. The \emph{shape} of $\Lambda$, denoted $\operatorname{sh}_\Lambda$, is its equivalence class under the action of rotations, reflections, and positive scalar dilations. The space of all such shapes is the double coset space

$$
\mathcal{S}_r := \operatorname{GL}_r(\mathbb{Z}) \backslash \operatorname{GL}_r(\mathbb{R}) / \operatorname{GO}_r(\mathbb{R}).
$$
\noindent Given a basis $\{b_1, \dots, b_r\}$ for $\Lambda$, we express each $b_i$ in terms of the standard basis as $b_i = \sum_{j=1}^r a_{i,j} e_j$. The matrix $(a_{i,j}) \in \operatorname{GL}_r(\mathbb{R})$ then determines a point in $\mathcal{S}_r$, namely the shape of $\Lambda$. This space is equipped with a natural measure $\mu_{\mathcal{S}_r}$ induced from Haar measure on $\operatorname{GL}_r(\mathbb{R})$.
\end{definition}
\noindent Given any measurable subset $Y \subset \mathcal{S}_r$, we write $\mu(Y) := \mu_{\mathcal{S}_r}(Y)$ for its measure.
\par It is often convenient to describe $\mathcal{S}_r$ as the quotient of a discrete group acting on a homogenous space $\cH_r$. Let $\mathcal{P}_r$ denote the space of positive definite $r \times r$ symmetric real matrices. The group $\operatorname{GL}_r(\mathbb{Z})$ acts on $\mathcal{P}_r$ by $G \mapsto M^T G M$, while $\lambda\in \mathbb{R}^\times$ acts by scaling, $G \mapsto \lambda^2 G$. Set $\cH_r:=\mathcal{P}_r/\mathbb{R}^\times$ and note that $\cH_r$ inherits a natural $\operatorname{GL}_r(\mathbb{Z})$-action. We identify $\cH_r$ with a subset of affine space $\mathbb{A}^{r^2}=\op{M}_r(\mathbb{R})$ as follows:
\[\cH_r=\{M\in \op{M}_r(\mathbb{R})\mid M\text{ is positive definite and symmetric with }\det M=1\}.\]There is a natural bijection $\operatorname{GL}_r(\mathbb{R}) / \operatorname{GO}_r(\mathbb{R}) \xrightarrow{\sim} \mathcal{H}_r$ given by $M \mapsto |\det M|^{-2/r} M^T M$, inducing an identification $\mathcal{S}_r \simeq \operatorname{GL}_r(\mathbb{Z}) \backslash \mathcal{H}_r$. This identification can be made explicit as follows. If $\Lambda$ is a lattice with basis $\{b_1, \dots, b_r\}$, its Gram matrix $\operatorname{Gr}(\Lambda)$ is defined by $(\langle b_i, b_j \rangle)_{i,j}$, and the shape of $\Lambda$ corresponds to the class of $\operatorname{Gr}(\Lambda)$ in $\mathcal{S}_r$. This class is invariant under change of basis by $M \in \operatorname{GL}_r(\mathbb{Z})$ and rescaling by $\lambda \in \mathbb{R}^\times$, as is seen from the following relation
$$
\operatorname{Gr}(\lambda M \cdot \Lambda) = \lambda^2 M^T \operatorname{Gr}(\Lambda) M,
$$
\noindent demonstrating that $\operatorname{sh}_\Lambda$ is indeed well defined.

\par The description of $\cH_2$ is quite simple, since it consists of matrices $\mtx{a}{b}{b}{c}\in \op{M}_2(\mathbb{R})$ with $a>0$ and $ac-b^2=1$. Let $\mathbb{H}$ denote the upper half plane in $\mathbb{C}$. Let $\op{GL}_2^+(\mathbb{R})$ be the subgroup of $\op{GL}_2(\mathbb{R})$ consisting of invertible matrices with positive determinant. There is a natural action of $\op{GL}_2(\mathbb{R})$ on $\mathbb{H}$ by fractional linear transformations:
\[\mtx{\alpha}{\beta}{\gamma}{\delta}(z):=\left(\frac{\alpha z+\beta}{\gamma z+\delta}\right).\] This action is transitive and the stabilizer of $i$ consists of matrices in $\op{GO}_2^+(\mathbb{R})$. This gives an identification 
\[\op{GL}_2(\mathbb{R})/\op{GO}_2(\mathbb{R})\xrightarrow{\sim}\mathbb{H}\] induced by the map $\mtx{\alpha}{\beta}{\gamma}{\delta}\mapsto \left(\frac{\alpha z+\beta}{\gamma z+\delta}\right)$. Consider the matrix
\[M=\mtx{1}{x}{0}{1}\mtx{y^{1/2}}{0}{0}{y^{-1/2}},\] where $y>0$ and let 
\[\mtx{a}{b}{b}{c}=MM^T=\mtx{(y+y^{-1}x^2)}{y^{-1}x}{y^{-1}x}{y^{-1}}.\]
\noindent There is a $\op{GL}_2(\mathbb{R})$-equivariant identification $\cH_2\xrightarrow{\sim} \mathbb{H}$ sending $\mtx{a}{b}{b}{c}$ to $z=x+iy\in \mathbb{H}$, where $x=\frac{b}{c}$ and $y=\sqrt{\frac{a}{c}-x^2}$. A fundamental domain for the action of $\op{GL}_2(\Z)$ on $\mathbb{H}$ is then given by \[\cF_2=\{x+iy\mid x\in [0, 1/2), \quad y>0, \quad x^2+y^2\geq 1\}.\] 





\par For $r > 2$, describing a fundamental domain for the action of $\operatorname{GL}_r(\mathbb{Z})$ on $\mathcal{H}_r$ becomes significantly more intricate. We focus here on the case $r = 3$, and give an explicit description of a fundamental domain for the action of $\operatorname{GL}_3(\mathbb{Z})$ on $\mathcal{H}_3$. Since $\operatorname{SL}_3(\mathbb{Z})$ is a subgroup of index two in $\operatorname{GL}_3(\mathbb{Z})$, with the two cosets distinguished by determinant $\pm 1$, it suffices to describe a fundamental domain for the action of $\operatorname{SL}_3(\mathbb{Z})$ on $\mathcal{H}_3$.

\par The resulting fundamental domain $\mathcal{F}_3$ can be viewed as a higher-rank generalization of the classical fundamental domain for the $\operatorname{SL}_2(\mathbb{Z})$-action on the upper half-plane. To realize $\mathcal{H}_3$ as an analogue of an upper half-space, we adopt Iwasawa coordinates, expressing each point $Y \in \mathcal{H}_3$ as
$$
Y = \begin{pmatrix}
1 & x_1 & x_2 \\
0 & 1 & x_3 \\
0 & 0 & 1
\end{pmatrix}^T\begin{pmatrix}
y_1 & 0 & 0 \\
0 & y_2 & 0 \\
0 & 0 & y_3
\end{pmatrix}
\begin{pmatrix}
1 & x_1 & x_2 \\
0 & 1 & x_3 \\
0 & 0 & 1
\end{pmatrix},
$$
\noindent where $y_j > 0$, $x_j \in \mathbb{R}$, and $\prod_{j=1}^3 y_j = 1$.

\par A fundamental domain for the action of $\operatorname{SL}_3(\mathbb{Z})$ on $\mathcal{H}_3$, due to Minkowski, is given by the region
$$
\mathcal{M}_3 = \left\{ Y = (y_{ij}) \in \mathcal{H}_3 \,\middle|\,
\begin{aligned}
& y_{11} \le y_{22} \le y_{33}, \quad 0 \le y_{12} \le \tfrac{1}{2} y_{11}, \\
& 0 \le y_{23} \le \tfrac{1}{2} y_{22}, \quad |y_{13}| \le \tfrac{1}{2} y_{11}, \\
& e^t Y e \ge y_{33} \quad \text{for all } e = (\pm 1, \pm 1, \pm 1)
\end{aligned}
\right\}.
$$
\noindent It follows that
\begin{equation}\label{FD1}
0 \leqslant x_1 \leqslant \frac{1}{2}, \quad\left|x_2\right| \leqslant \frac{1}{2}, \quad 0 \leqslant y_1 x_1 x_2+y_2 x_3 \leqslant \frac{1}{2}\left(y_1 x_1^2+y_2\right),  
\end{equation}
and
\begin{equation}\label{FD2}
   -\frac{1}{4} y_1 / y_2 \leqslant x_3 \leqslant \frac{1}{2}+(3 / 8) y_1 / y_2. 
\end{equation}
\noindent
Hence, we get 
$$\cF_3 = \{ v =(x_1, x_2, x_3, y_1, y_2) \in \mathbb{R}^5 \ \ \mid \ \  v \ \text{ satisfies \eqref{FD1} and \eqref{FD2}} \ \}.$$



\medskip
\par Now we discuss the shape of a number field $K$ with $[K:\Q]=n$ and ring of integers $\cO_K$. Let $ \{1, \alpha_1, \dots, \alpha_{n-1}\} $ be an integral basis. Let $ \{\sigma_i\}_{1 \leq i \leq n} $ denote the set of embeddings of $ K $ into $ \mathbb{C} $. The Minkowski embedding  
\[
\begin{aligned}
    j: K &\to \mathbb{C}^{n} \\
    \alpha &\mapsto \big( \sigma_1(\alpha), \dots, \sigma_n(\alpha) \big)
\end{aligned}
\]
identifies $ K $ with an $ n $-dimensional subspace of $ \mathbb{C}^n $. The real span of the image, denoted $ K_{\mathbb{R}} $, inherits a natural inner-product structure, and the restriction of $ j $ to $ \mathcal{O}_{K} $ yields a rank-$ n $ lattice in $ K_{\mathbb{R}} $. Although one might be inclined to define the shape of $ K $ as the shape of this lattice, such a definition presents difficulties when studying the distribution of shapes across a family of number fields. This is because each lattice necessarily contains the vector $ j(1) $, this introduces a nontrivial constraint, disrupting randomness in the distribution of shapes. To circumvent this issue, we instead define the shape of $ K $ in terms of a sublattice obtained by projecting $ j(\mathcal{O}_K) $ onto the orthogonal complement of $ j(1) $. This is accomplished via the map  
\[
\alpha^{\perp} := n\alpha - \operatorname{tr}(\alpha),
\]
which sends elements of $ \mathcal{O}_K $ to those of trace zero. The set $ \mathcal{O}_K^{\perp} $, consisting of the images of $ \mathcal{O}_K $ under this map, forms a rank $r:=(n-1)$ submodule of $ \mathcal{O}_K $. 
\begin{definition}
    We define the shape of $K$, denoted $\op{sh}_K$, to be the shape of the lattice $j(\cO_K^\perp)$ in $\cS_{n-1}$.
\end{definition}
Given a basis $ \{1, \alpha_1, \dots, \alpha_{n-1}\} $ of $ \mathcal{O}_K $, the corresponding basis for $ \mathcal{O}_K^{\perp} $ is given by $\{\alpha_1^{\perp}, \dots, \alpha_{n-1}^{\perp} \} $. As we discussed above, the Gram matrix is given as follows: 
\begin{equation}\label{gram matrix computation}
\operatorname{Gr}\left(j\left(\mathcal{O}_{K}^{\perp}\right)\right) = \big( \langle j(\alpha_i^{\perp}), j(\alpha_j^{\perp}) \rangle \big).
\end{equation}

\section{Integral basis and shapes of pure quartic number fields}\label{sec3}
\par Let $m \in \Z$ be fourth-power free, and let $\beta_m$ be a root of $X^4 - m \in \Q[X]$, with $\arg(\beta_m) = 0$ if $m > 0$, and $\arg(\beta_m) = \pi/4$ if $m < 0$. The field $K_m:=\Q(\beta_m)$ is called a \emph{pure quartic field}. Up to replacing $m$ by $m^3$, we assume in this section that $
m = abc^3,
$ where:
\begin{enumerate}
\item[(i)] $a \ne 1$, and $a, b, c$ are pairwise coprime, squarefree integers;
\item[(ii)] $b, c > 0$;
\item[(iii)] $|a| \ge c$ if $a$ is odd;
\item[(iv)] $c$ is odd;
\item[(v)] $m \ne -4$.
\end{enumerate}
\noindent Condition (v) excludes the reducible case $X^4 + 4$. The assumptions (iii) and (iv) are in place in order to compute an integral basis and thus describe the shape of $K_m$. Since isomorphic number fields have the same shape, we shall in the next section replace (i)-(v) above with the following:
\begin{enumerate}
\item[(i')] $a \ne 1$, and $a, b, c$ are pairwise coprime, squarefree integers;
\item[(ii')] $b, c > 0$;
\item[(iii')] $|a| \ge c$;
\item[(iv')] $m \ne -4$.
\end{enumerate}

We henceforth assume $m$ is in this normal form, and define
$$
\alpha = \frac{\beta_m^2}{bc}, \quad \beta = \beta_m, \quad \gamma = \frac{\beta_m^3}{bc^2}.
$$
\noindent An explicit integral basis for $K_m = \Q(\beta_m)$ was computed by Funakura \cite{Funakura}, as recalled below.

\begin{theorem}[Funakura]\label{IB funakura}
Given an integer $m$ as above, $\mathcal{B}=\{1,\lambda,\mu,\nu\}$ given below forms an integral basis of $K_m$:
$$
\begin{aligned}
&\begin{array}{|c|c|c|c|c|}
\hline m & 1 & \lambda & \mu & \nu \\
\hline 1(\bmod 8) & 1 & \frac{1+\alpha}{2} & \beta & \frac{a b+\alpha+b \beta+\gamma}{4} \\
\hline \begin{array}{c}
2(\bmod 4) \\
3(\bmod 4)
\end{array} & 1 & \alpha & \beta & \gamma \\
\hline \begin{array}{c}
4(\bmod 16) \\
5(\bmod 8)
\end{array} & 1 & \frac{1+\alpha}{2} & \beta & \frac{\beta+\gamma}{2} \\
\hline 12(\bmod 32) & 1 & \alpha & \frac{1+\alpha+\beta}{2} & \frac{\beta+\gamma}{2} \\
\hline 28(\bmod 32) & 1 & \alpha & \frac{1+\alpha+\beta}{2} & \frac{4 \alpha+b \beta+2 \gamma}{8} .\\
\hline
\end{array}
\end{aligned}
$$
\end{theorem}
\begin{proof}
    See \cite[Theorem 1]{Funakura}.
\end{proof}
\noindent We shall denote the integral basis by $\mathcal{B}_\ast$ when $K_m$ is of Type $\ast$. As a corollary to the above result, Funakura computes the discriminant $\Delta_m$ of $K_m$.
\begin{corollary}[Funakura]\label{funakura corollary}
    With respect to notation above, we have that
    \[\Delta_m=\begin{cases}
        -2^2a^3 b^2 c^3 &\text{ if }m\equiv 1\mod{8}, \quad 28\mod{32},\\
        -2^4a^3 b^2 c^3 &\text{ if }m\equiv 4\mod{16},\quad 5\mod{8},\quad 12\mod{32},\\
        -2^8a^3 b^2 c^3 &\text{ if }m\equiv 2,3\mod{4}.
    \end{cases}\]
\end{corollary}
\begin{proof}
    The above result is \cite[Corollary 1]{Funakura}.
\end{proof}
\par Theorem \ref{IB funakura} provides an explicit integral basis in all cases under consideration, since $m$ is not divisible by 8. For each such case, we now determine the shape of the corresponding lattice $K_m$ by computing the Gram matrix. The table above suggests that the cases $m \equiv 2,\,3 \pmod{4}$ admit particularly simple bases. We begin by computing the Gram matrix in this setting and subsequently use change-of-basis matrices to treat the remaining cases.

\par If $m>0$, $K_m$ has two real embeddings. Let $\sigma_1$ (resp. $\sigma_2$) be the embedding $K\hookrightarrow \mathbb{C}$ sending $\beta$ to $\beta$ (resp. $-\beta$). Let $\tau$ be the complex embedding mapping $\beta\mapsto i \beta$. On the other hand, when $m<0$, let $\tau_1$ (resp. $\tau_2$) be the embedding defined by $\beta\mapsto i\beta$ (resp. $\beta\mapsto -\beta$). Let $j_1$ and $j_2$ denote the embeddings of $K$ into $\mathbb{C}^4$ corresponding to the cases $m > 0$ and $m < 0$, respectively. These are defined by

$$
j_1(a) = \big(\sigma_1(a), \sigma_2(a), \tau(a), \overline{\tau(a)}\big), \quad \text{and} \quad 
j_2(a) = \big(\tau_1(a), \tau_2(a), \overline{\tau_2(a)}, \overline{\tau_1(a)}\big).
$$
\noindent Let $\mathcal{P}_{\ast, \pm}$ denote the Gram matrix when $m$ is of Type $\ast$ and sign $\pm$.
\begin{proposition}
The Gram matrix of the basis $ \mathcal{B}_{\rm{II}} = \{1,\alpha,\beta,\gamma\}$ is given by :

$$
\mathcal{P}_{\rm{II}, +}=  
\begin{pmatrix}
4 & 0 & 0 & 0 \\
0 & 4\alpha^2 & 0 & 0 \\
0 & 0 & 4\beta^2 & 0 \\
0 & 0 & 0 & 4\gamma^2
\end{pmatrix},\ \
\mathcal{P}_{\rm{II}, -}=  
\begin{pmatrix}
4 & 0 & 0 & 0 \\
0 & 4|\alpha|^2 & 0 & 0 \\
0 & 0 & 4|\beta|^2 & 0 \\
0 & 0 & 0 & 4|\gamma|^2
\end{pmatrix}.
$$

\end{proposition}

\begin{proof}
\par Let $m > 0$. Then we have:
\begin{enumerate}
\item $j_1(1) = (1,1,1,1)$,
\item $j_1(\alpha) = (\alpha, \alpha, -\alpha, -\alpha)$,
\item $j_1(\beta) = (\beta, -\beta, i\beta, -i\beta)$,
\item $j_1(\gamma) = (\gamma, -\gamma, -i\gamma, i\gamma)$.
\end{enumerate}
Taking the standard Hermitian inner product on $\mathbb{C}^4$, we compute:

$$
\langle j_1(\lambda), j_1(\lambda) \rangle = 4\lambda^2
$$
\noindent for each $\lambda \in \{1, \alpha, \beta, \gamma\}$, and
$$
\langle j_1(\lambda_1), j_1(\lambda_2) \rangle = 0
$$
\noindent for all distinct $\lambda_1, \lambda_2 \in \{1, \alpha, \beta, \gamma\}$. That is, the vectors $j_1(1), j_1(\alpha), j_1(\beta), j_1(\gamma)$ form an orthogonal set, and the corresponding Gram matrix is diagonal with entries $4\lambda^2$ on the diagonal, as claimed.

\medskip

Now suppose $m < 0$. Then:
\begin{enumerate}
\item $j_2(1) = (1, 1, 1, 1)$,
\item $j_2(\beta) = (\beta, i\beta, -i\bar{\beta}, \bar{\beta})$,
\item $j_2(\alpha) = \left( \frac{\beta^2}{bc}, -\frac{\beta^2}{bc}, -\frac{\bar{\beta}^2}{bc}, \frac{\bar{\beta}^2}{bc} \right)$,
\item $j_2(\gamma) = \left( \frac{\beta^3}{bc^2}, i\frac{\beta^3}{bc^2}, -i\frac{\bar{\beta}^3}{bc^2}, \frac{\bar{\beta}^3}{bc^2} \right)$.
\end{enumerate}
\noindent It is easy to see that the vectors $j_2(\alpha), j_2(\beta), j_2(\gamma)$ are mutually orthogonal with respect to the standard Hermitian inner product. The squared norm of $j_2(\alpha)$ is:

$$
\begin{aligned}
\langle j_2(\alpha), j_2(\alpha) \rangle 
&= \left\langle \left( \frac{\beta^2}{bc}, -\frac{\beta^2}{bc}, -\frac{\bar{\beta}^2}{bc}, \frac{\bar{\beta}^2}{bc} \right), \left( \frac{\beta^2}{bc}, -\frac{\beta^2}{bc}, -\frac{\bar{\beta}^2}{bc}, \frac{\bar{\beta}^2}{bc} \right) \right\rangle \\
&= \frac{1}{(bc)^2} \left( |\beta^2|^2 + |\beta^2|^2 + |\bar{\beta}^2|^2 + |\bar{\beta}^2|^2 \right) = 4 \left( \frac{|\beta|^2}{bc} \right)^2 = 4|\alpha|^2.
\end{aligned}
$$
\noindent A similar computation gives:

$$
\langle j_2(\beta), j_2(\beta) \rangle = 4|\beta|^2, \quad \langle j_2(\gamma), j_2(\gamma) \rangle = 4|\gamma|^2.
$$

Hence, the Gram matrix of the set $\{j_2(1), j_2(\alpha), j_2(\beta), j_2(\gamma)\}$ is diagonal with entries $4|\lambda|^2$, and the vectors are pairwise orthogonal as desired.

\end{proof}

\begin{proposition}\label{1 propn}
The Gram matrix of the basis $ \mathcal{B}_{\rm{I}} = \{1, \frac{1 +\alpha}{2}, \beta, \frac{ab + \alpha + b \beta +\gamma}{4} \}$ is given by :
$$
\mathcal{P}_{\rm{I}, +}=  
\begin{pmatrix}
4 & 2 & 0 & ab \\
2 & 1 + \alpha^2 & 0 & \frac{ab + \alpha^2}{2} \\
0 & 0 & 4\beta^2 & b\beta^2 \\
ab & \frac{ab + \alpha^2}{2} & b\beta^2 & \frac{a b^2 + b \alpha^2 + b^2 \beta^2 + \gamma^2}{4}
\end{pmatrix},  \ \
\mathcal{P}_{\rm{I}, -} = \begin{pmatrix}
4 & 2 & 0 & ab \\
2 & 1 + |\alpha|^2 & 0 & \frac{ab + |\alpha|^2}{2} \\
0 & 0 & 4|\beta|^2 & b|\beta|^2 \\
ab & \frac{ab + |\alpha|^2}{2} & b|\beta|^2 & \frac{a b^2 + b |\alpha|^2 + b^2 |\beta|^2 + |\gamma|^2}{4} 
\end{pmatrix}.
$$
\end{proposition}
\begin{proof}
    The change of basis matrix from $\mathcal{B}_{\rm{II}}$ to $\mathcal{B}_{\rm{I}}$ is given as
    $$
    C_{\rm{II}}:= 
    \begin{pmatrix}
1 & 0 & 0 & 0 \\
1/2 & 1/2 & 0 & 0 \\
0 & 0 & 1 & 0 \\
ab/4 & 1/4 & b/4 & 1/4
\end{pmatrix}.
    $$
    Therefore, 
    \begin{align*}
    \mathcal{P}_{\rm{I}, +} = C_{\rm{II}} \, \mathcal{P}_{\rm{II}, +} \, C_{\rm{II}}^T &= 
    \begin{pmatrix}
1 & 0 & 0 & 0 \\
1/2 & 1/2 & 0 & 0 \\
0 & 0 & 1 & 0 \\
ab/4 & 1/4 & b/4 & 1/4
\end{pmatrix}
\begin{pmatrix}
4 & 0 & 0 & 0 \\
0 & 4\alpha^2 & 0 & 0 \\
0 & 0 & 4\beta^2 & 0 \\
0 & 0 & 0 & 4\gamma^2
\end{pmatrix}
 \begin{pmatrix}
1 & 1/2 & 0 & ab/4 \\
0 & 1/2 & 0 & 1/4 \\
0 & 0 & 1 &  b/4 \\
0 & 0 &0 & 1/4
\end{pmatrix} \\ 
&= \begin{pmatrix}
4 & 2 & 0 & ab \\
2 & 1 + \alpha^2 & 0 & \frac{ab + \alpha^2}{2} \\
0 & 0 & 4\beta^2 & b\beta^2 \\
ab & \frac{ab + \alpha^2}{2} & b\beta^2 & \frac{a b^2 + b \alpha^2 + b^2 \beta^2 + \gamma^2}{4}
\end{pmatrix}.
    \end{align*}

\noindent
Similarly, 
 
    \begin{align*}
    \mathcal{P}_{\rm{I}, -} = C_{\rm{II}} \, \mathcal{P}_{\rm{II}, -} \, C_{\rm{II}}^T &= \begin{pmatrix}
4 & 2 & 0 & ab \\
2 & 1 + |\alpha|^2 & 0 & \frac{ab + |\alpha|^2}{2} \\
0 & 0 & 4|\beta|^2 & b|\beta|^2 \\
ab & \frac{ab + |\alpha|^2}{2} & b|\beta|^2 & \frac{a b^2 + b |\alpha|^2 + b^2 |\beta|^2 + |\gamma|^2}{4}
\end{pmatrix}.
    \end{align*}
\end{proof}

\begin{proposition}\label{4.5 propn}
The Gram matrix of the basis $\mathcal{B}_{\rm{III}} = \{1,\frac{1 + \alpha}{2},\beta,\frac{\beta + \gamma}{2}\}$ is given by :

$$
\mathcal{P}_{\rm{III}, +}=  
\begin{pmatrix} 4 & 2 & 0 & 0 \\ 2 & 1 + \alpha^2 & 0 & 0 \\ 0 & 0 & 4\beta^2 & 2\beta^2 \\ 0 & 0 & 2\beta^2 & \beta^2 + \gamma^2 \end{pmatrix}, \ \ 
\mathcal{P}_{\rm{III}, -} =  
\begin{pmatrix} 4 & 2 & 0 & 0 \\ 
2 & 1 + |\alpha|^2 & 0 & 0 \\ 
0 & 0 & 4|\beta|^2 & 2|\beta|^2 \\ 
0 & 0 & 2|\beta|^2 & |\beta|^2 + |\gamma|^2 
\end{pmatrix}
$$
\end{proposition}

\begin{proof}
    The change of basis matrix from $\mathcal{B}_{\rm{II}}$ to $\mathcal{B}_{\rm{III}}$ is given as
    $$
    C_{\rm{III}}:= 
    \begin{pmatrix}
1 & 0 & 0 & 0 \\
1/2 & 1/2 & 0 & 0 \\
0 & 0 & 1 & 0 \\
0 & 0 & 1/2 & 1/2
\end{pmatrix}.
    $$
    Therefore, 
    \begin{align*}
    \mathcal{P}_{\rm{III},+}= C_{\rm{III}} \, \mathcal{P}_{\rm{II},+} \, C_{\rm{III}}^T &= 
     \begin{pmatrix}
1 & 0 & 0 & 0 \\
1/2 & 1/2 & 0 & 0 \\
0 & 0 & 1 & 0 \\
0 & 0 & 1/2 & 1/2
\end{pmatrix}
\begin{pmatrix}
4 & 0 & 0 & 0 \\
0 & 4\alpha^2 & 0 & 0 \\
0 & 0 & 4\beta^2 & 0 \\
0 & 0 & 0 & 4\gamma^2
\end{pmatrix}
 \begin{pmatrix}
1 & 1/2 & 0 & 0 \\
0 & 1/2 & 0 & 0 \\
0 & 0 & 1 &  1/2 \\
0 & 0 &0 & 1/2
\end{pmatrix} \\ 
&= \begin{pmatrix} 4 & 2 & 0 & 0 \\ 2 & 1 + \alpha^2 & 0 & 0 \\ 0 & 0 & 4\beta^2 & 2\beta^2 \\ 0 & 0 & 2\beta^2 & \beta^2 + \gamma^2 \end{pmatrix}.
    \end{align*}
A similar relation holds when $m$ is negative.
\end{proof}

\begin{proposition}\label{12 propn}
The Gram matrix of the basis $\mathcal{B}_{\rm{IV}} = \{1,\alpha,\frac{1 + \alpha +\beta}{2}, \frac{\beta + \gamma}{2} \}$ is given by :

$$
\mathcal{P}_{\rm{IV}, +}=  
\begin{pmatrix} 
4 & 0 & 2 & 0 \\
0 & 4\alpha^2 & 2\alpha^2 & 0 \\
2 & 2\alpha^2 & 1 + \alpha^2 + \beta^2 & \beta^2 \\
0 & 0 & \beta^2 & \gamma^2 
\end{pmatrix}, \ \
\mathcal{P}_{\rm{IV}, -}= 
\begin{pmatrix} 
4 & 0 & 2 & 0 \\
0 & 4|\alpha|^2 & 2|\alpha|^2 & 0 \\
2 & 2|\alpha|^2 & 1 + |\alpha|^2 + |\beta|^2 & |\beta|^2 \\
0 & 0 & |\beta|^2 & |\gamma|^2 
\end{pmatrix}.
$$
\end{proposition}
\begin{proof}
    The change of basis matrix from $\mathcal{B}_{\rm{II}}$ to $\mathcal{B}_{\rm{IV}}$ is given as
    $$
    C_{\rm{IV}}:= 
    \begin{pmatrix}
1 & 0 & 0 & 0 \\
0 & 1 & 0 & 0 \\
1/2 & 1/2 & 1/2 & 0 \\
0 & 0 & 1/2 & 1/2
\end{pmatrix}.
    $$
    Therefore,
     \begin{align*}
    \mathcal{P}_{\rm{IV}, +}= C_{\rm{IV}} \, \mathcal{P}_{\rm{II}, +} \, C_{\rm{IV}}^T &= 
     \begin{pmatrix}
1 & 0 & 0 & 0 \\
0 & 1 & 0 & 0 \\
1/2 & 1/2 & 1/2 & 0 \\
0 & 0 & 1/2 & 1/2
\end{pmatrix}
\begin{pmatrix}
4 & 0 & 0 & 0 \\
0 & 4\alpha^2 & 0 & 0 \\
0 & 0 & 4\beta^2 & 0 \\
0 & 0 & 0 & 4\gamma^2
\end{pmatrix}
 \begin{pmatrix}
1 & 0 & 1/2 & 0 \\
0 & 1 & 1/2 & 0 \\
0 & 0 & 1/2 &  1/2 \\
0 & 0 & 0 & 1/2
\end{pmatrix} \\ 
&= \begin{pmatrix} 4 & 0 & 2 & 0 \\ 0 & 4\alpha^2 & 2\alpha^2 & 0 \\ 2 & 2\alpha^2 & 1 + \alpha^2 + \beta^2 & \beta^2 \\ 0 & 0 & \beta^2 & \gamma^2 \end{pmatrix}.
    \end{align*}
A similar relation holds when $m$ is negative.
\end{proof}

\begin{proposition}\label{28 propn}
The Gram matrix of the basis $\mathcal{B}_{\rm{V}} = \{1,\alpha,\frac{1 + \alpha + \beta}{2},\frac{4\alpha + b \beta + 2 \gamma}{8}\}$ is given by :

$$
\mathcal{P}_{\rm{V}, +}=  
\begin{pmatrix} 
4 & 0 & 2 & 0 \\
0 & 4\alpha^2 & 2\alpha^2 & 2\alpha^2 \\ 
2 & 2\alpha^2 & 1 + \alpha^2 + \beta^2 & \alpha^2 + \frac{b}{4}\beta^2 \\
0 & 2\alpha^2 & \alpha^2 + \frac{b}{4}\beta^2 & \alpha^2 + \frac{b^2}{64}\beta^2 + \frac{1}{4}\gamma^2 
\end{pmatrix}
$$
and
$$
\mathcal{P}_{\rm{V}, -}=  
\begin{pmatrix} 
4 & 0 & 2 & 0 \\
0 & 4|\alpha|^2 & 2|\alpha|^2 & 2|\alpha|^2 \\ 
2 & 2|\alpha|^2 & 1 + |\alpha|^2 + |\beta|^2 & |\alpha|^2 + \frac{b}{4}|\beta|^2 \\
0 & 2|\alpha|^2 & |\alpha|^2 + \frac{b}{4}|\beta|^2 & |\alpha|^2 + \frac{b^2}{64}|\beta|^2 + \frac{1}{4}|\gamma|^2 
\end{pmatrix}.
$$
\end{proposition}
\begin{proof}
    The change of basis matrix from $\mathcal{B}_{\rm{II}}$ to $\mathcal{B}_{\rm{V}}$ is given as
    $$
    C_{\rm{V}}:= 
    \begin{pmatrix}
1 & 0 & 0 & 0 \\
0 & 1 & 0 & 0 \\
1/2 & 1/2 & 1/2 & 0 \\
0 & 1/2 & b/8 & 1/4
\end{pmatrix}.
    $$
    Therefore,
     \begin{align*}
    \mathcal{P}_{\rm{V}, +} = C_{\rm{V}} \, \mathcal{P}_{\rm{II}, +} \, C_{\rm{V}}^T &= 
      \begin{pmatrix}
1 & 0 & 0 & 0 \\
0 & 1 & 0 & 0 \\
1/2 & 1/2 & 1/2 & 0 \\
0 & 1/2 & b/8 & 1/4
\end{pmatrix}
\begin{pmatrix}
4 & 0 & 0 & 0 \\
0 & 4\alpha^2 & 0 & 0 \\
0 & 0 & 4\beta^2 & 0 \\
0 & 0 & 0 & 4\gamma^2
\end{pmatrix}
 \begin{pmatrix}
1 & 0 & 1/2 & 0 \\
0 & 1 & 1/2 & 1/2 \\
0 & 0 & 1/2 &  b/8 \\
0 & 0 & 0 & 1/4
\end{pmatrix} \\ 
&= \begin{pmatrix} 4 & 0 & 2 & 0 \\ 0 & 4\alpha^2 & 2\alpha^2 & 2\alpha^2 \\ 2 & 2\alpha^2 & 1 + \alpha^2 + \beta^2 & \alpha^2 + \frac{b}{4}\beta^2 \\ 0 & 2\alpha^2 & \alpha^2 + \frac{b}{4}\beta^2 & \alpha^2 + \frac{b^2}{64}\beta^2 + \frac{1}{4}\gamma^2 \end{pmatrix}.
    \end{align*}
    A similar relation holds when $m$ is negative.
\end{proof}

We summarize the results above.

\begin{proposition}\label{main prop section 3}
    When $K_m$ is of Type $\ast$ and $m$ has sign $\pm$, there is an invertible matrix $C_\ast$ such that 
    \[\mathcal{P}_{\ast, \pm}=C_\ast\mathcal{P}_{\rm{II}, \pm}C_\ast^T.\] Here, the matrix $C_\ast$ is given by:
    \[\begin{split}
        & C_{\rm{I}}=  \begin{pmatrix}
1 & 0 & 0 & 0 \\
1/2 & 1/2 & 0 & 0 \\
0 & 0 & 1 & 0 \\
ab/4 & 1/4 & b/4 & 1/4
\end{pmatrix},\\
& C_{\rm{II}}=  \begin{pmatrix}
1 & 0 & 0 & 0 \\
0 & 1 & 0 & 0 \\
0 & 0 & 1 & 0 \\
0 & 0 & 0 & 1
\end{pmatrix},\\
& C_{\rm{III}}=\begin{pmatrix}
1 & 0 & 0 & 0 \\
1/2 & 1/2 & 0 & 0 \\
0 & 0 & 1 & 0 \\
0 & 0 & 1/2 & 1/2
\end{pmatrix}, \\
& C_{\rm{IV}}=\begin{pmatrix}
1 & 0 & 0 & 0 \\
0 & 1 & 0 & 0 \\
1/2 & 1/2 & 1/2 & 0 \\
0 & 0 & 1/2 & 1/2
\end{pmatrix},\\
& C_{\rm{V}}= \begin{pmatrix}
1 & 0 & 0 & 0 \\
0 & 1 & 0 & 0 \\
1/2 & 1/2 & 1/2 & 0 \\
0 & 1/2 & b/8 & 1/4
\end{pmatrix}.
    \end{split}\]
\end{proposition}

\begin{proof}
    The result follows from Propositions \ref{1 propn}, \ref{4.5 propn}, \ref{12 propn} and \ref{28 propn} (and their proofs).
\end{proof}

\begin{proof}[Proof of Theorem \ref{shape complete invariant}]
By Proposition~\ref{main prop section 3}, the shape of $K_m$ lies in one of five distinct loci, determined by the Type of $m$. These sets are disjoint since distinct torus orbits in the space of shapes are disjoint. Thus if
\[
\operatorname{sh}(K_{m_1})=\operatorname{sh}(K_{m_2}),
\]
then $m_1$ and $m_2$ must have the same type. Write 
\[
m_i = a_i b_i^2 c_i^3,
\qquad i=1,2,
\]
where $a_i,b_i,c_i$ are pairwise coprime squarefree integers with $b_i>0$ and $c_i>0$.  Equality of shapes implies
\[
\lambda_1(m_1)=\lambda_1(m_2), \qquad 
\lambda_2(m_1)=\lambda_2(m_2),
\]
and therefore
\[
\frac{c_1}{|a_1|}=\frac{c_2}{|a_2|}, 
\qquad 
b_1=b_2.
\]
Thus $a_1 c_2 = \pm a_2 c_1$. Since $(a_i,c_i)=1$ for $i=1,2$, it follows that
\[
a_1=\pm a_2,\qquad b_1=b_2,\qquad c_1=c_2.
\]

If $a_1=a_2$, then $m_1=m_2$ and we are done. Suppose instead that $a_1=-a_2$, so that $m_1=-m_2$. Because $m_1$ and $m_2=-m_1$ have the same Type (as their shapes coincide), the sign change does not move $m_1$ out of its Type. By inspection, this occurs only when $m_1 \equiv 2 \pmod{4}$. This proves the theorem.
\end{proof}

\section{Distribution results}\label{sec4}
\par For the sake of discussion, assume that $m>0$ and $m\equiv 2, 3\pmod{4}$. The shape $\op{sh}_{m}$ of  $j(\cO_{K_m})^\perp\in \mathcal{S}_3$, is given by \[\op{sh}_m=\dmtx{\alpha^2}{\beta^2}{\gamma^2}=\dmtx{\beta^4/(bc)^2}{\beta^2}{\beta^6/(bc^2)^2}.\] After permuting the variables, we have that:
\begin{equation*}\op{sh}_m=\dmtx{\gamma_1^2}{\gamma_2^2}{\gamma_3^2}\end{equation*}
where $\gamma_1:=\beta$, $\gamma_2:=\beta^2/bc$ and $\gamma_3:=\beta^3/bc^2$.
We scale this matrix by dividing by $(abc)$. Let $\lambda_i=\gamma_i^2/(abc)$, we find that $\op{sh}_m=\op{diag}\left(\lambda_1, \lambda_2, \lambda_3\right)$, 
\[\begin{split}
& \lambda_1^2= \frac{\beta^4}{(abc)^2}=\frac{ab^2c^3}{a^2b^2c^2}=\frac{c}{a},\\
& \lambda_2= \frac{\beta^4}{(bc)^2(abc)}=\frac{ab^2c^3}{ab^3c^3}=\frac{1}{b},\\
& \lambda_3^2=\frac{\beta^{12}}{(bc^2)^4(abc)^2}=\frac{a^3b^6c^9}{a^2b^{6}c^{10}}=\frac{a}{c}.
\end{split}\]
Since $\beta$ is real and $a,b,c>0$, we find that $\lambda_i>0$ for $i=1, 2, 3$, and therefore,
\[\lambda_1=\sqrt{\frac{c}{a}},\quad \lambda_2=\frac{1}{b}\quad \text{and}\quad \lambda_3=\sqrt{\frac{a}{c}}.\] We find that \[\op{sh}_m=\dmtx{\sqrt{\frac{c}{a}}}{\frac{1}{b}}{\sqrt{\frac{a}{c}}},\] we refer to the pair \begin{equation}\label{sh-parameter}\lambda(m)=(\lambda_1(m), \lambda_2(m))=\left(\sqrt{\frac{c}{a}}, \frac{1}{b}\right)\end{equation} as the \emph{shape parameter}. 
\par In the more general case, suppose that $m$ is of Type $\ast$. Assume that the conditions (i')-(iv') are satisfied and set 
\[\lambda(m)=(\lambda_1(m), \lambda_2(m))=\left(\sqrt{\frac{c}{|a|}}, \frac{1}{b}\right).\]
The shape of $K_m$ is then given by the projection of the matrix
$$
C_\ast \begin{pmatrix}
1 & 0 & 0 & 0 \\
0 & \sqrt{c/|a|} & 0 & 0 \\
0 & 0 & 1/b & 0 \\
0 & 0 & 0 & \sqrt{|a|/c}
\end{pmatrix} C_\ast^T
$$
\noindent to the orthogonal complement of the vector $(1, 1, 1, 1)$.
\par Let $\mathcal{R} \subset \mathbb{R}^2$ be a closed and bounded region whose boundary is rectifiable and has total length $L$. The result below is often attributed to Lipschitz, with an early reference appearing in the work of Bachmann \cite{bachmann1904analytische}, who credits Lipschitz as the source (see also \cite{Davenportlemma}).

\begin{lemma}[Lipschitz Principle]\label{davenport}
Let $\mathcal{R} \subset \mathbb{R}^2$ be as above. Then the number of integer lattice points contained in $\mathcal{R}$ satisfies
$$
\left| \#(\mathcal{R} \cap \mathbb{Z}^2) - \operatorname{Area}(\mathcal{R}) \right| \leq 4(L + 1).
$$
\end{lemma}
\par Consider the torus orbit $\mathcal{T} \subset \mathcal{H}_3$ consisting of diagonal matrices of the form
$$
D(x_1, x_2) = x_2^{1/3} \begin{pmatrix}
x_1^{-1} &  &  \\
 & x_2^{-1} &  \\
 &  & x_1
\end{pmatrix}.
$$
\noindent Define
$$
\psi(x) := \sum_{n = 1}^{\lfloor x \rfloor} \alpha(n),
\quad \text{where} \quad
\alpha(n) := \begin{cases}
n^{-2/3} \displaystyle\prod_{\ell \mid n} \left( \dfrac{\ell - 1}{\ell + 1} \right) & \text{if } n \text{ is squarefree}, \\
0 & \text{otherwise},
\end{cases}
$$
\noindent and the product is over prime divisors $\ell$ of $n$. Let $\hat{\mu}$ denote the measure $\frac{\alpha(\lfloor x_2 \rfloor)}{x_1}\, dx_1\, dx_2$ on $\mathcal{T}$.

\par Given a pair $(R_1, R_2)$, define the region

$$
W(R_1, R_2) := \left\{ D(x_1, x_2) \in \mathcal{T} \;\middle|\; x_1^2 \in [1, R_1], \; x_2 \in [1, R_2] \right\}.
$$
\noindent Then we find that
$$
\hat{\mu}\left(W(R_1, R_2)\right) = \int_1^{R_2} \int_1^{\sqrt{R_1}} \frac{\alpha(\lfloor x_2 \rfloor)}{x_1} \, dx_1 \, dx_2 = \frac{1}{2} \log R_1 \cdot \psi(R_2).
$$

\noindent Let $\mathcal{R}(N, R_1, R_2)$ be defined as follows:
\[\mathcal{R}(N, R_1, R_2):=\{(a,b,c)\in \mathbb{R}^3_{\geq 1}\mid ab^{2/3}c<N, \quad  \frac{a}{c}\in [1,R_1]\quad \text{and }\quad b\in [1, R_2]\}.\]

\noindent Let $\mathcal{R}_{\mathbb{Z}^3}(N, R_1, R_2) := \mathcal{R}(N, R_1, R_2) \cap \mathbb{Z}^3$. For $M, R \geq 1$, define
$$
\mathcal{S}(M, R) := \left\{ (a, c) \in \mathbb{R}_{\geq 1}^2 \,\middle|\, ac < M,\; \frac{a}{c} \in [1, R] \right\},
$$
\noindent and let $\mathcal{S}_{\mathbb{Z}^2}(M, R) := \mathcal{S}(M, R) \cap \mathbb{Z}^2$. Then we have the identity
\begin{equation}\label{mathcal R to mathcal S eqn}
\mathcal{R}_{\mathbb{Z}^3}(N, R_1, R_2) = \bigsqcup_{b=1}^{\lfloor R_2\rfloor} \mathcal{S}_{\mathbb{Z}^2}\left(\frac{N}{b^{2/3}}, R_1\right),
\end{equation}
where the union is disjoint.

\begin{lemma}\label{S(M,R) area lemma}
    The area of $\mathcal{S}(M, R)$ is given by: \[
\op{Area}\left(\mathcal{S}(M, R)\right)=\begin{cases}
 \frac{M}{2}\log R -\frac{(R-1)}{2} &\text{ if }M\geq R;\\
 \frac{M}{2}\log M-\frac{(M-1)}{2} & \text{ if }M<R.\\
\end{cases}\]
\end{lemma}
\begin{proof}Note that for any fixed $c \in [1, \sqrt{M}]$, the condition $ac < M$ and $\frac{a}{c} \in [1, R]$ together imply that
\[a \in \left[c, \min(Rc, M/c)\right].\]
\noindent We begin by analyzing the case $M \geq R$. Accordingly, we split the region of integration into two parts:

$$
\begin{aligned}
\operatorname{Area}(\mathcal{S}(M, R)) &= \int_{1}^{\sqrt{M/R}} \int_{c}^{Rc} \, da\, dc + \int_{\sqrt{M/R}}^{\sqrt{M}} \int_{c}^{M/c} \, da\, dc \\
&= \int_{1}^{\sqrt{M/R}} (R - 1)c \, dc + \int_{\sqrt{M/R}}^{\sqrt{M}} \left( \frac{M}{c} - c \right) dc \\
&= \frac{R - 1}{2} \left(\frac{M}{R}\right) - \frac{1}{2}\left(M - \frac{M}{R}\right) + \frac{M}{2} \log R \\
&= \frac{M}{2} \log R - \frac{R - 1}{2}.
\end{aligned}
$$

Next, suppose $M < R$. In this case, for all $c \in [1, \sqrt{M}]$, we have $M/c \leq Rc$, so the upper bound of the inner integral is simply $M/c$. Thus,

$$
\begin{aligned}
\operatorname{Area}(\mathcal{S}(M, R)) &= \int_{1}^{\sqrt{M}} \int_{c}^{M/c} \, da\, dc = \int_{1}^{\sqrt{M}} \left( \frac{M}{c} - c \right) dc \\
&= M \log \sqrt{M} - \frac{1}{2}(M - 1) = \frac{M}{2} \log M - \frac{M - 1}{2}.
\end{aligned}
$$
\end{proof}
\begin{lemma}\label{L computation lemma}
    Let $L=L(M, R)$ be the length of the boundary of $\mathcal{S}(M, R)$. Then, we have that
    \[L\leq \left(\sqrt{2}+\sqrt{1+R^2}\right)(\sqrt{M}-1)+(R-1).\]
\end{lemma}

\begin{proof}
Assume that $M \geq R$. Then the total boundary length is given by $
L = L_1 + L_2 + L_3 + L_4$, where the individual contributions are as follows:
\begin{itemize}
    \item $L_1=\op{length}\{(c, c) \mid c \in [1, \sqrt{M}]\} = \sqrt{2}(\sqrt{M} - 1)$,
    \item $
L_2 = \text{length}\{(c, Rc) \mid c \in [1, \sqrt{M/R}]\} = \sqrt{1 + R^2}(\sqrt{M/R} - 1)
$,
\item $
L_3 = \op{length} \{(c, M/c) \mid c \in [\sqrt{M/R}, \sqrt{M}]\} = \int_{\sqrt{M/R}}^{\sqrt{M}} \sqrt{1 + \left( \frac{M}{c^2} \right)^2} \, dc\leq \sqrt{1+R^2}\left(\sqrt{M}-\sqrt{M/R}\right)
$,
\item $L_4=R-1$ is the length of the vertical segment from $(1,1)$ to $(1,R)$.
\end{itemize}
\noindent Thus, the total length is given by
 \[L\leq \left(\sqrt{2}+\sqrt{1+R^2}\right)(\sqrt{M}-1)+(R-1).\]
 \noindent Similar analysis applies in the case when $M<R$ to give the result.
 \end{proof}

\begin{proposition}
We have that
    \[\# \mathcal{R}_{\Z^3}(N, R_1, R_2)= \frac{N}{2}\log R_1\left(\sum_{n=1}^{\lfloor R_2\rfloor}n^{-2/3}\right)+O_{R_1, R_2}(\sqrt{N}). \] 
\end{proposition}

\begin{proof}
From Lemma \ref{davenport} and Lemmas \ref{S(M,R) area lemma} and \ref{L computation lemma}, it follows that
\[\left|\# \left(\mathcal{S}(M, R)\cap \Z^2\right)-\frac{M}{2}\log R\right|\leq 4(L+1)\leq  C R \sqrt{M},\] where $C>0$ is an absolute constant.
Noting that 
\[\# \mathcal{R}_{\Z^3}(N, R_1, R_2)=\sum_{b=1}^{\lfloor R_2\rfloor} \# \mathcal{S}_{\Z^2}(N/b^{2/3}, R_1),\]
we deduce that 
\[\left|\# \mathcal{R}_{\Z^3}(N, R_1, R_2)-\frac{N}{2}\left(\sum_{b=1}^{\lfloor R_2\rfloor}b^{-2/3}\right) \log R_1 \right|\leq C R_1\sqrt{R_2} \sqrt{N}.\]
\end{proof}
\begin{remark}
Higher-dimensional analogues of Lipschitz’s principle, due to Davenport \cite{Davenportlemma}, allow for effective volume estimates by bounding the lower-dimensional projections (or "shadows") of a region containing the lattice points; see, for example, \cite[Corollary 6.13]{Hol22}. However, in our situation, the area of the two-dimensional projection onto the $(a,c)$-plane is too large to yield meaningful results.
\end{remark}
\begin{definition}\label{carefree definition}
Let $\ell$ be a prime number. A triple $(a,b,c)\in \Z_{>0}^3$ is said to be $\ell$-carefree if
$$
\ell^2 \nmid ab,\quad \ell^2 \nmid bc,\quad \ell^2 \nmid ca.
$$
\noindent The triple $(a,b,c)$ is said to be carefree if each of the products $ab$, $bc$, and $ca$ is squarefree. Equivalently, $(a,b,c)$ is carefree if and only if it is $\ell$-carefree for all primes $\ell$.
\end{definition}
\noindent Given a positive integer $n$, we say that a triple $(a,b,c) \in \Z_{>0}^3$ is $n$-carefree if it is $\ell$-carefree for all primes $\ell$ dividing $n$. Fix a residue class $\tau$ modulo $32$ such that $8\nmid \tau$. Let $\mathcal{C}^\tau(N, R_1, R_2)$ (respectively, $\mathcal{C}_n^\tau(N, R_1, R_2)$) denote the set of triples $(a,b,c) \in \mathcal{R}_{\Z^3}(N, R_1, R_2)$ that are carefree (respectively, $n$-carefree) and $ab^2c^3\equiv \tau\pmod{32}$. Then
$$
\mathcal{C}^\tau(N, R_1, R_2) = \bigcap_{\ell} \mathcal{C}_\ell^\tau(N, R_1, R_2),
$$
\noindent where the intersection is taken over all prime numbers $\ell$.

Let $n>0$ be squarefree and odd. Define $\Omega_n$ to be the set of residue classes $(\bar{a}, \bar{b}, \bar{c}) \in (\Z/32n^2\Z)^3$ such that for each prime $\ell$ dividing $n$, one has that $\ell^2 \nmid \bar{a} \bar{b},\quad \ell^2 \nmid \bar{b} \bar{c},\quad \ell^2 \nmid \bar{c} \bar{a}
$, and moreover, $\bar{a}\bar{b}^2\bar{c}^3\equiv \tau\pmod{32}$. Given an odd prime $\ell$, let $\mathcal{A}_\ell$ denote $(\bar{a}, \bar{b}, \bar{c}) \in (\Z/\ell^2\Z)^3$ such that $\ell^2 \nmid \bar{a} \bar{b},\quad \ell^2 \nmid \bar{b} \bar{c},\quad \ell^2 \nmid \bar{c} \bar{a}
$. Denote by $\mathcal{A}_2^\tau$ the set of $(\bar{a}, \bar{b}, \bar{c}) \in (\Z/32\Z)^3$ such that $4\nmid \bar{a}$ and $\bar{a}\bar{b}^2\bar{c}^3\equiv \tau\pmod{32}$. Clearly by the Chinese Remainder Theorem,
$$
\Omega_n = \mathcal{A}_2^\tau\times \prod_{\ell \mid n} \mathcal{A}_\ell,
$$
where $\ell$ ranges over prime divisors of $n$. Define $$
\mathfrak{d}_\ell := \frac{\# \mathcal{A}_\ell}{\# (\Z/\ell^2\Z)^3} = \frac{\# \mathcal{A}_\ell}{\ell^6},
$$ if $\ell$ is an odd prime. On the other hand, set 
\[\mathfrak{d}_2^\tau := \frac{\# \mathcal{A}_2^\tau}{\# (\Z/32\Z)^3} = \frac{\# \mathcal{A}_2^\tau}{32^3}.\] For each squarefree natural number $b$, define $n_\tau(b)$ to be the number of pairs $(\bar{a}, \bar{c})\in (\Z/32\Z)^2$ such that $\tau=\bar{a}\bar{b}^2\bar{c}^3$.

\begin{lemma}
For any odd prime $\ell$, we have
$$
\# \mathcal{A}_\ell = (\ell^2-\ell)^2\bigl(\ell^2+2\ell-3\bigr), \qquad \text{and hence} \qquad \mathfrak{d}_\ell = \left(1-\frac{1}{\ell}\right)^2
\left(1+\frac{2}{\ell}-\frac{3}{\ell^2}\right).
$$

\end{lemma}

\begin{proof} The number of tuples $(\bar{a},\bar{b},\bar{c})\in(\Z/\ell^2\Z)^3$ such that
$\ell\nmid \bar{a}\bar{b}\bar{c}$ is $(\ell^2-\ell)^3$, since each coordinate
must be a unit modulo~$\ell$.  
The number of triples for which exactly one coordinate is divisible by~$\ell$
is
\[
3\cdot \ell(\ell-1)\,(\ell^2-\ell)^2,
\]
where the factor $\ell(\ell-1)$ counts elements divisible by $\ell$ but not by
$\ell^2$ in $\Z/\ell^2\Z$, and the factor $3$ accounts for the choice of the
coordinate.

In total, the number of triples $(\bar{a},\bar{b},\bar{c})$ for which at most one
coordinate is divisible by~$\ell$ is
\[
(\ell^2-\ell)^3
\;+\;
3\ell(\ell-1)(\ell^2-\ell)^2
=
(\ell^2-\ell)^2\bigl(\ell^2+2\ell-3\bigr).
\]
Dividing by $\ell^6$, we obtain
\[
\mathfrak{d}_\ell=\frac{(\ell^2-\ell)^2(\ell^2+2\ell-3)}{\ell^6}
=
\left(1-\frac{1}{\ell}\right)^2
\left(1+\frac{2}{\ell}-\frac{3}{\ell^2}\right).
\]
\end{proof}

\begin{proposition}\label{Cn estimate}
    Let $n$ be an odd squarefree natural number which is divisible by all primes $\ell\leq R_2$ and $\tau$ be a residue class mod $32$ such that $8\nmid \tau$. Then for $N>n^4 R_1 R_2$,
    \[\# \cC_n^\tau (N, R_1, R_2)=\frac{N}{2\times 32^2}\prod_{\ell|n} \left(1-\frac{3}{\ell^2}+\frac{2}{\ell^3}\right)\log R_1\sum_{\substack{b=1,\\
    b\,\square\text{-free}}}^{\lfloor R_2\rfloor}  b^{-2/3}n_\tau(b)\prod_{\ell|b}\left(\frac{\ell}{\ell+2}\right)+O_{n, R_1, R_2}(\sqrt{N}),\] where $O_{n,R_1, R_2}(\cdot)$ indicates that the constant in the error bound depends on $R_1$, $R_2$ and $n$. 
\end{proposition}
\begin{proof}
    Let $\Lambda_n$ be the lattice $(32 n^2 \Z)^3\subset \Z^3$. Note that $\Omega_n\subset \Z^3/\Lambda_n$ choose a set of representatives $\widetilde{\Omega}_n\subset ([0,32 n^2)\cap \Z)^3$. We find that 
    \[\begin{split}\mathcal{C}_n^\tau(N, R_1, R_2)=&\bigsqcup_{z\in \widetilde{\Omega}_n} (z+\Lambda_n)\cap \mathcal{R}(N, R_1, R_2),\\
    =&\bigsqcup_{z\in \widetilde{\Omega}_n} \Lambda_n \cap \left(\mathcal{R}(N, R_1, R_2)-z\right).
    \end{split}\]
    Observe that 
    \[\# \left(\Lambda_n \cap \left(\mathcal{R}(N, R_1, R_2)-z\right)\right)=\# \left(\Z^3 \cap \left(\frac{1}{32 n^2}\cdot \mathcal{R}(N, R_1, R_2)-\frac{z}{32 n^2}\right)\right)\]
    and we estimate \[
    \begin{split}& \# \left(\Z^3 \cap \left(\frac{1}{32 n^2}\cdot \mathcal{R}(N, R_1, R_2)-\frac{z}{32 n^2}\right)\right)\\
    =&\# \left\{(a-z_1/32 n^2,b-z_2/32 n^2,c-z_3/32 n^2)\in \Z_{\geq 1}^3\mid ab^{2/3}c<\frac{N}{(32n^2)^{8/3}},\, a/c\in [1, R_1]\,\text{ and } b\in \left[\frac{-1}{32 n^2}, \frac{R_2}{32 n^2}\right]\right\}\\
    =&\sum_{b'=0 }^{\lfloor (R_2-z_2)/32n^2\rfloor} \# \left\{(a-z_1/32n^2, c-z_3/32n^2)\in \Z_{\geq 1}^2\mid ac<\frac{N}{32^2n^4(32n^2b'+z_2)^{2/3}},\, a/c\in [1,R_1]\right\}\\
    =& \sum_{b'=0 }^{\lfloor (R_2-z_2)/32^2n^4\rfloor} \#\left(\mathcal{S}\left( \frac{N}{32^2n^4(32n^2b'+z_2)^{2/3}},R_1\right)-\left(\frac{z_1}{32n^2}, \frac{z_3}{32n^2}\right)\right)\cap \Z^2 \\
    =& \sum_{b'=0 }^{\lfloor (R_2-z_2)/32n^2\rfloor} \op{Area}\left(\mathcal{S}\left( \frac{N}{32^2n^4(32n^2b'+z_2)^{2/3}},R_1\right)\right)+O\left(\sum_{b'=0 }^{\lfloor (R_2+z_2)/32n^2\rfloor} L\left( \frac{N}{32^2n^4(32n^2b'+z_2)^{2/3}},R_1\right)\right),
    \end{split}\] 
    where in the above chain of equalities, we set $b':=b-z_2/32n^2$. Note that in the sum, $b$ does not take integer values, but $b'$ does. From Lemma \ref{S(M,R) area lemma}, we find that the area of $\mathcal{S}(M, R)$ is given by: \[
\op{Area}\left(\mathcal{S}(M, R)\right)=\begin{cases}
 \frac{M}{2}\log R -\frac{(R-1)}{2} &\text{ if }M\geq R;\\
 \frac{M}{2}\log M-\frac{(M-1)}{2} & \text{ if }M<R.\\
\end{cases}\]
Since $N>n^4R_1R_2$, we have that 
\[\begin{split}
    &\# \left(\Z^3 \cap \left(\frac{1}{32n^2}\cdot \mathcal{R}(N, R_1, R_2)-\frac{z}{32n^2}\right)\right) \\
    =& \frac{N\log R_1}{2}\frac{1}{32^2n^4}\left\{\sum_{\substack{b=1,\\
    b\equiv z_2\mod{32 n^2}}}^{\lfloor R_2\rfloor}  b^{-2/3}\right\}+O_{n, R_1, R_2}(\sqrt{N}) \\
\end{split}\]
Thus we find that 
\[\begin{split}
    \#\mathcal{C}_n^\tau(N, R_1, R_2)=& \frac{N\log R_1}{2}\frac{1}{32^2n^4}\sum_{\substack{b=1,\\
    b\,\square\text{-free}}}^{\lfloor R_2\rfloor}  \frac{\#\{z\in \widetilde{\Omega}_n\mid z_2=b\pmod{32n^2}\}}{b^{2/3}}+O_{n, R_1, R_2}(\sqrt{N})\\
    =& \frac{N\log R_1 }{2 \times 32^2}\frac{1}{n^4}\sum_{\substack{b=1,\\
    b\,\square\text{-free}}}^{\lfloor R_2\rfloor}  \frac{n_\tau(b)\times\prod_{\ell|b}(\ell^2-\ell)^2\times \prod_{\ell|n, \ell\nmid b} \left(\ell^4-3\ell^2+2\ell\right)}{b^{2/3}}+O_{n, R_1, R_2}(\sqrt{N})\\
     =& \frac{N }{2\times 32^2}\prod_{\ell|n} \left(1-\frac{3}{\ell^2}+\frac{2}{\ell^3}\right)\log R_1\sum_{\substack{b=1,\\
    b\,\square\text{-free}}}^{\lfloor R_2\rfloor}  n_\tau(b)b^{-2/3}\prod_{\ell|b}\left(\frac{\ell}{\ell+2}\right)+O_{n, R_1, R_2}(\sqrt{N}).\\
\end{split}\]
\end{proof}
\noindent Let $\cC_\ell'(N, R_1, R_2)$ be the complement of $\cC_\ell(N, R_1, R_2)$ in $\mathcal{R}_{\Z^3}(N, R_1, R_2)$.

\begin{lemma}\label{lemma 4.1}
    We have that 
    \[\frac{\# \cC_\ell'(N, R_1, R_2)}{N}=O\left(\frac{1}{\ell^{4/3}}+\frac{1}{\sqrt{N}}\right),\] where the implied constant depends only on $R_1$ and $R_2$, and not on $N$ or $\ell$.
\end{lemma}
\begin{proof}
Suppose $(a,b,c)\in \cC_\ell'(N, R_1, R_2)$, then there are the following cases to consider: 
\begin{enumerate}
    \item $\ell^2|a$, 
    \item $\ell^2|b$, 
    \item $\ell^2|c$, 
    \item
    $\ell|a$ and $\ell|b$, 
    \item $\ell|a$ and $\ell|c$, 
    \item $\ell|b$ and $\ell|c$.
\end{enumerate}
For $i=1, \dots, 6$, let $\cF_i(N, R_1, R_2)$ be the subset of $\cC_\ell'(N, R_1, R_2)$ consisting of $(a,b,c)$ satisfying $(i)$ above. First we bound $\#\cF_1(N, R_1, R_2)$ we write $a=\ell^2 a'$. Note that $(a',b,c)$ lies in the set 
\[\{(a',b,c)\in \Z_{\geq 1}^3\mid a'b^{2/3}c\leq N/\ell^2,\quad a'/c\in [1/\ell^2, R_1/\ell^2],\quad b\in [1, R_2]\}.\]
We find that
\[\begin{split}\# \cF_1(N, R_1, R_2)=& \sum_{b=1}^{\lfloor R_2\rfloor} \# \{(a',c)\in \Z_{\geq 1}^2\mid a'c\leq \frac{N}{\ell^2b^{2/3}},\quad a'/c\in [1/\ell^2, R_1/\ell^2]\}\\
=& \sum_{b=1}^{\lfloor R_2\rfloor} O\left( \op{Area}\left(\{(a',c)\in \Z_{\geq 1}^2\mid a'c\leq \frac{N}{\ell^2b^{2/3}},\quad a'/c\in [1/\ell^2, R_1/\ell^2]\}\right)\right)\\
+ & \sum_{b=1}^{\lfloor R_2\rfloor} O\left( L\left(\{(a',c)\in \Z_{\geq 1}^2\mid a'c\leq \frac{N}{\ell^2b^{2/3}},\quad a'/c\in [1/\ell^2, R_1/\ell^2]\}\right)\right)\\
=& O_{R_1, R_2}(N/\ell^2),
\end{split}\]
where the area and length calculation is similar to that in Lemmas \ref{S(M,R) area lemma} and \ref{L computation lemma}.
We find that 
\[\begin{split}\# \cF_2(N, R_1, R_2)=& \sum_{b'=1}^{\lfloor R_2/\ell^2\rfloor} \# \{(a,c)\in \Z_{\geq 1}^2\mid ac\leq \frac{N}{\ell^{4/3}b'^{2/3}},\quad a/c\in [1, R_1]\}\\
=& \sum_{b'=1}^{\lfloor R_2/\ell^2\rfloor} O\left( \op{Area}\left(\{(a,c)\in \Z_{\geq 1}^2\mid ac\leq \frac{N}{\ell^{4/3}b'^{2/3}},\quad a/c\in [1, R_1]\}\right)\right)\\
+ & \sum_{b'=1}^{\lfloor R_2/\ell^2\rfloor} O\left( L\left(\{(a,c)\in \Z_{\geq 1}^2\mid ac\leq \frac{N}{\ell^{4/3}b'^{2/3}},\quad a/c\in [1, R_1]\}\right)\right)\\
=& O_{R_1, R_2}\left(\frac{N}{\ell^{4/3}}\right) + O_{R_1, R_2}(\sqrt{N}), \\
\# \cF_3(N, R_1, R_2)=& \sum_{b=1}^{\lfloor R_2\rfloor} \# \{(a,c')\in \Z_{\geq 1}^2\mid ac'\leq N/\ell^2b,\quad a/c'\in [1, \ell^2 R_1]\}\\
=& O_{R_1, R_2}(N \log \ell/\ell^2)+ O_{R_1, R_2}(\sqrt{N}),\\
\# \cF_4(N, R_1, R_2)=& \sum_{b'=1}^{\lfloor R_2/\ell \rfloor} \# \{(a',c)\in \Z_{\geq 1}^2\mid a'c\leq N/\ell^{5/3}b',\quad a'/c\in [1,  R_1/\ell]\}\\
=& O_{R_1, R_2}(N /\ell^{5/3})+ O_{R_1, R_2}(\sqrt{N}),\\
\# \cF_5(N, R_1, R_2)=& \sum_{b=1}^{\lfloor R_2 \rfloor} \# \{(a',c')\in \Z_{\geq 1}^2\mid a'c'\leq N/\ell^{2}b,\quad a'/c'\in [1, R_1]\}\\
=& O_{R_1, R_2}(N /\ell^2)+ O_{R_1, R_2}(\sqrt{N}),\\ 
\# \cF_6(N, R_1, R_2)=& \sum_{b'=1}^{\lfloor R_2/\ell \rfloor} \# \{(a,c')\in \Z_{\geq 1}^2\mid ac'\leq N/\ell^{5/3} b',\quad a/c'\in [1, \ell R_1]\}\\
=& O_{R_1, R_2}\left(N \log \ell /\ell^{5/3}\right)+ O_{R_1, R_2}(\sqrt{N}).\\ 
\end{split}\]
Thus putting it all together, one finds that 
\[\frac{\# \cC_\ell'(N, R_1, R_2)}{N}=\sum_{i=1}^6 \frac{\cF_i(N, R_1, R_2)}{N}=O\left(\ell^{-4/3}+\frac{1}{\sqrt{N}}\right).\]
\end{proof}

\begin{lemma}
    Let $Y>0$ be a positive constant and let $\mathcal{D}_Y(N, R_1, R_2):=\bigcup_{\ell>Y} \cC_\ell'(N, R_1, R_2)$. Then we find that 
    \[\frac{\# \mathcal{D}_Y(N, R_1, R_2)}{N}=O\left(\sum_{\ell>Y} \ell^{-4/3}+ \frac{1}{\log N}\right).\]
\end{lemma}
\begin{proof}
    Clearly $\cC_\ell'(N, R_1, R_2)=\emptyset$ unless $\ell\ll \sqrt{N}$. Thus from Lemma \ref{lemma 4.1},
     \[\frac{\# \mathcal{D}_Y(N, R_1, R_2)}{N}=O\left(\sum_{\ell>Y} \ell^{-4/3}+ \frac{\pi(\sqrt{N})}{\sqrt{N}}\right).\]
    The result follows from the prime number theorem.
\end{proof}

\begin{proposition}\label{main propn of article}
We have that
    \[\begin{split}\lim_{N\rightarrow \infty}\frac{\# \cC^\tau (N, R_1, R_2)}{N}=& \frac{1}{2048}\prod_\ell\left(1-\frac{3}{\ell^2}+\frac{2}{\ell^3}\right)\log R_1\sum_{\substack{b=1,\\
    b\,\square\text{-free}}}^{\lfloor R_2\rfloor}  n_\tau(b)b^{-2/3}\prod_{\ell|b}\left(\frac{\ell}{\ell+2}\right),\\ 
    =& \frac{1}{2048}\prod_{\ell}\left(1-\frac{3}{\ell^2}+\frac{2}{\ell^3}\right)\log R_1\psi_\tau(R_2),\end{split}\]
    where
    \[\psi_\tau(x):=\sum_{n=1}^{\lfloor x\rfloor} n_\tau(n)\alpha(n).\]

\end{proposition}
\begin{proof}
Let $n$ be a squarefree natural number which is divisible by all primes $\ell\leq R_2$. Then for $N>n^4 R_1 R_2$, Proposition \ref{Cn estimate} asserts that
    \[\# \cC_n^\tau(N, R_1, R_2)=\frac{N}{2048}\prod_{\ell|n} \left(1-\frac{3}{\ell^2}+\frac{2}{\ell^3}\right)\log R_1\sum_{\substack{b=1,\\
    b\,\square\text{-free}}}^{\lfloor R_2\rfloor}  n_\tau(b)b^{-2/3}\prod_{\ell|b}\left(\frac{\ell}{\ell+2}\right)+O_{n, R_1, R_2}(\sqrt{N}).\] 
For $Y>0$, set $n_Y:=\prod_{\ell\leq Y} \ell$. We find that 
     \[\limsup_{N\rightarrow \infty} \frac{\# \mathcal{C}^\tau(N, R_1, R_2)}{N}\leq \lim_{N\rightarrow \infty} \frac{\# \cC_{n_Y}^\tau(N, R_1, R_2)}{N}=\frac{1}{2048}\prod_{\ell\leq Y} \left(1-\frac{3}{\ell^2}+\frac{2}{\ell^3}\right)\log R_1\sum_{\substack{b=1,\\
    b\,\square\text{-free}}}^{\lfloor R_2\rfloor}  n_\tau(b)b^{-2/3}\prod_{\ell|b}\left(\frac{\ell}{\ell+2}\right).\]
     Taking the limit of the right hand side as $Y\rightarrow \infty$, we find that 
     \[\limsup_{N\rightarrow \infty} \frac{\# \mathcal{C}^\tau(N, R_1, R_2)}{N}\leq \frac{1}{2048}\prod_\ell\left(1-\frac{3}{\ell^2}+\frac{2}{\ell^3}\right)\log R_1\sum_{\substack{b=1,\\
    b\,\square\text{-free}}}^{\lfloor R_2\rfloor}  n_\tau(b)b^{-2/3}\prod_{\ell|b}\left(\frac{\ell}{\ell+2}\right).\]
     On the other hand, since 
     \[\cC_{n_Y}^\tau(N, R_1, R_2)\subseteq \mathcal{C}^\tau(N, R_1, R_2)\cup \mathcal{D}_Y(N, R_1, R_2), \] we find that
      \[\begin{split}\liminf_{N\rightarrow \infty} \frac{\# \mathcal{C}^\tau(N, R_1, R_2)}{N} & \geq \lim_{N\rightarrow \infty} \frac{\# \cC_{n_Y}(N, R_1, R_2)}{N}-\limsup_{N\rightarrow \infty} \frac{\# \mathcal{D}_Y(N, R_1, R_2)}{N} \\ 
      & \geq  \frac{1}{2048}\prod_{\ell\leq Y} \left(1-\frac{3}{\ell^2}+\frac{2}{\ell^3}\right)\log R_1\sum_{\substack{b=1,\\
    b\,\square\text{-free}}}^{\lfloor R_2\rfloor}  n_\tau(b)b^{-2/3}\prod_{\ell|b}\left(\frac{\ell}{\ell+2}\right) - C\sum_{\ell>Y} \ell^{-4/3}\end{split}\]
      for some positive constant $C>0$. Taking the limit of the right hand side with $Y
      \rightarrow \infty$, we find that 
       \[\liminf_{N\rightarrow \infty} \frac{\# \mathcal{C}^\tau(N, R_1, R_2)}{N}\geq \frac{1}{2048}\prod_\ell\left(1-\frac{3}{\ell^2}+\frac{2}{\ell^3}\right)\log R_1\sum_{\substack{b=1,\\
    b\,\square\text{-free}}}^{\lfloor R_2\rfloor}  n_\tau(b)b^{-2/3}\prod_{\ell|b}\left(\frac{\ell}{\ell+2}\right).\]
       Therefore, we have shown that 
        \[\lim_{N\rightarrow \infty} \frac{\# \mathcal{C}^\tau(N, R_1, R_2)}{N}= \frac{1}{2048}\prod_\ell\left(1-\frac{3}{\ell^2}+\frac{2}{\ell^3}\right)\log R_1\sum_{\substack{b=1,\\
    b\,\square\text{-free}}}^{\lfloor R_2\rfloor}  n_\tau(b)b^{-2/3}\prod_{\ell|b}\left(\frac{\ell}{\ell+2}\right).\]
\end{proof}

\noindent We now give a proof of the main result. 
\begin{proof}[Proof of Theorem \ref{Main theorem}]
  Assume that $m > 0$ and is of Type $\mathrm{II}$. The argument in the remaining cases proceeds similarly. Write $m = ab^2c^3$, where $(a,b,c)$ is a carefree triple with $a, b, c > 0$, and suppose further that $c/a > 1$. This condition ensures that the isomorphism class of the number field $K_m$ is uniquely determined. Recall from Corollary~\ref{funakura corollary} that the discriminant of $K_m$ is given by $\Delta_m = -2^8 a^3 b^2 c^3$. Setting $X = 2^8 N^3$, the condition $|\Delta_m| \leq X$ becomes equivalent to
$$
ab^{2/3}c \leq N.
$$
\noindent Consider the four rectangles:
\[
\begin{split}
\mathcal{R}_1 &:= [1,R_1] \times [1,R_2], \\
\mathcal{R}_2 &:= [1,R_1] \times [1,R_2'], \\
\mathcal{R}_3 &:= [1,R_1'] \times [1,R_2], \\
\mathcal{R}_4 &:= [1,R_1'] \times [1,R_2'],
\end{split}\]
\noindent so that by inclusion–exclusion,
$$
\int_{\mathcal{R}} = \int_{\mathcal{R}_1} - \int_{\mathcal{R}_2} - \int_{\mathcal{R}_3} + \int_{\mathcal{R}_4}.
$$
\noindent It therefore suffices to prove the theorem under the simplifying assumption $R_1' = R_2' = 1$, in which case the claim follows directly from Proposition~\ref{main propn of article}.
\end{proof}

\bibliographystyle{alpha}
\bibliography{references}

\end{document}